\documentclass[12pt]{amsart}       
\usepackage{txfonts}
\usepackage{amssymb}
\usepackage{eucal}
\usepackage{graphicx}
\usepackage{amsmath}
\usepackage{amscd}
\usepackage[all]{xy}           
\usepackage{amsfonts,latexsym}
\usepackage{xspace}
\usepackage{epsfig}
\usepackage{float}
\usepackage{needspace}
\usepackage{color}
\usepackage{fancybox}
\usepackage{multicol}
\usepackage{colordvi}

\usepackage{ifpdf}
\ifpdf
\usepackage[colorlinks,final,backref=page,hyperindex]{hyperref}
\else
\usepackage[colorlinks,final,backref=page,hyperindex,hypertex]{hyperref}
\fi

\usepackage[active]{srcltx} 
\usepackage{tikz}

\usepackage{mathrsfs}
\newtheorem{theorem}{Theorem}[section]
\newtheorem{prop}[theorem]{Proposition}
\newtheorem{lemma}[theorem]{Lemma}
\newtheorem{coro}[theorem]{Corollary}
\newtheorem{prop-def}{Proposition-Definition}[section]
\newtheorem{coro-def}{Corollary-Definition}[section]

\theoremstyle{definition}
\newtheorem{defn}[theorem]{Definition}
\newtheorem{remark}[theorem]{Remark}

\newtheorem{exam}[theorem]{Example}

\newcommand{\nc}{\newcommand}
\nc{\tred}[1]{\textcolor{red}{#1}}
\nc{\tblue}[1]{\textcolor{blue}{#1}}
\nc{\tgreen}[1]{\textcolor{green}{#1}}
\nc{\tpurple}[1]{\textcolor{purple}{#1}}
\nc{\btred}[1]{\textcolor{red}{\bf #1}}
\nc{\btblue}[1]{\textcolor{blue}{\bf #1}}
\nc{\btgreen}[1]{\textcolor{green}{\bf #1}}
\nc{\btpurple}[1]{\textcolor{purple}{\bf #1}}
\nc{\NN}{{\mathbb N}}
\nc{\ncsha}{{\mbox{\cyr X}^{\mathrm NC}}} \nc{\ncshao}{{\mbox{\cyr
X}^{\mathrm NC}_0}}

\newcommand{\efootnote}[1]{}

\newcommand{\delete}[1]{}

\delete{
\nc{\mlabel}[1]{\label{#1}}  
\nc{\mcite}[1]{\cite{#1}}  
\nc{\mref}[1]{\ref{#1}}  
\nc{\mbibitem}[1]{\bibitem{#1}} 
}

\nc{\mlabel}[1]{\label{#1}{\hfill \hspace{1cm}{\bf{{\ }\hfill(#1)}}}}
\nc{\mcite}[1]{\cite{#1}{{\bf{{\ }(#1)}}}}  
\nc{\mref}[1]{\ref{#1}{{\bf{{\ }(#1)}}}}  
\nc{\mbibitem}[1]{\bibitem[\bf #1]{#1}} 

\nc{\opa}{\ast} \nc{\opb}{\odot} \nc{\op}{\bullet} \nc{\pa}{\frakL}
\nc{\arr}{\rightarrow} \nc{\lu}[1]{(#1)} \nc{\mult}{\mrm{mult}}
\nc{\diff}{\mathfrak{Diff}}
\nc{\opc}{\sharp}\nc{\opd}{\natural}
\nc{\ope}{\circ}
\nc{\dpt}{\mathrm{d}}
\nc{\hck}{H_{RT}}
\nc{\vdf}{\calf}
\nc{\ldf}{\calf_\ell}
\nc{\hlf}{H_\ell}
\nc{\onek}{\mathbf{1}_\bfk}

\nc{\tforall}{\quad \text{ for all }}

\nc{\AW}{\mathcal{A}}

\nc{\rba}{Rota-Baxter algebra\xspace}
\nc{\rbas}{Rota-Baxter algebras\xspace}

\nc{\ari}{\mathrm{ar}}

\nc{\lef}{\mathrm{lef}}

\nc{\Sh}{\mathrm{ST}}

\nc{\Cr}{\mathrm{Cr}}

\nc{\st}{{Schr\"oder tree}\xspace}
\nc{\sts}{{Schr\"oder trees}\xspace}

\nc{\vertset}{\Omega} 

\nc{\assop}{\quad \begin{picture}(5,5)(0,0)
\line(-1,1){10}
\put(-2.2,-2.2){$\bullet$}
\line(0,-1){10}\line(1,1){10}
\end{picture} \quad \smallskip}

\nc{\operator}{\begin{picture}(5,5)(0,0)
\line(0,-1){6}
\put(-2.6,-1.8){$\bullet$}
\line(0,1){9}
\end{picture}}

\nc{\idx}{\begin{picture}(6,6)(-3,-3)
\put(0,0){\line(0,1){6}}
\put(0,0){\line(0,-1){6}}
\end{picture}}

\nc{\pb}{{\mathrm{pb}}}
\nc{\Lf}{{\mathrm{Lf}}}

\nc{\lft}{{left tree}\xspace}
\nc{\lfts}{{left trees}\xspace}

\nc{\fat}{{fundamental averaging tree}\xspace}

\nc{\fats}{{fundamental averaging trees}\xspace}
\nc{\avt}{\mathrm{Avt}}

\nc{\rass}{{\mathit{RAss}}}

\nc{\aass}{{\mathit{AAss}}}

\nc{\vin}{{\mathrm Vin}}    
\nc{\lin}{{\mathrm Lin}}    
\nc{\inv}{\mathrm{I}n}
\nc{\gensp}{V} 
\nc{\genbas}{\mathcal{V}} 
\nc{\bvp}{V_P}     
\nc{\gop}{{\,\omega\,}}     

\nc{\bin}[2]{ (_{\stackrel{\scs{#1}}{\scs{#2}}})}  
\nc{\binc}[2]{ \left (\!\! \begin{array}{c} \scs{#1}\\
    \scs{#2} \end{array}\!\! \right )}  
\nc{\bincc}[2]{  \left ( {\scs{#1} \atop
    \vspace{-1cm}\scs{#2}} \right )}  
\nc{\bs}{\bar{S}} \nc{\cosum}{\sqsubset} \nc{\la}{\longrightarrow}
\nc{\rar}{\rightarrow} \nc{\dar}{\downarrow} \nc{\dprod}{**}
\nc{\dap}[1]{\downarrow \rlap{$\scriptstyle{#1}$}}
\nc{\md}{\mathrm{dth}} \nc{\uap}[1]{\uparrow
\rlap{$\scriptstyle{#1}$}} \nc{\defeq}{\stackrel{\rm def}{=}}
\nc{\disp}[1]{\displaystyle{#1}} \nc{\dotcup}{\
\displaystyle{\bigcup^\bullet}\ } \nc{\gzeta}{\bar{\zeta}}
\nc{\hcm}{\ \hat{,}\ } \nc{\hts}{\hat{\otimes}}
\nc{\barot}{{\otimes}} \nc{\free}[1]{\bar{#1}}
\nc{\uni}[1]{\tilde{#1}} \nc{\hcirc}{\hat{\circ}} \nc{\lleft}{[}
\nc{\lright}{]} \nc{\lc}{\lfloor} \nc{\rc}{\rfloor}
\nc{\curlyl}{\left \{ \begin{array}{c} {} \\ {} \end{array}
    \right .  \!\!\!\!\!\!\!}
\nc{\curlyr}{ \!\!\!\!\!\!\!
    \left . \begin{array}{c} {} \\ {} \end{array}
    \right \} }
\nc{\longmid}{\left | \begin{array}{c} {} \\ {} \end{array}
    \right . \!\!\!\!\!\!\!}
\nc{\onetree}{\bullet} \nc{\ora}[1]{\stackrel{#1}{\rar}}
\nc{\ola}[1]{\stackrel{#1}{\la}}
\nc{\ot}{\otimes} \nc{\mot}{{{\boxtimes\,}}}
\nc{\otm}{\overline{\boxtimes}} \nc{\sprod}{\bullet}
\nc{\scs}[1]{\scriptstyle{#1}} \nc{\mrm}[1]{{\rm #1}}
\nc{\margin}[1]{\marginpar{\rm #1}}   
\nc{\dirlim}{\displaystyle{\lim_{\longrightarrow}}\,}
\nc{\invlim}{\displaystyle{\lim_{\longleftarrow}}\,}
\nc{\mvp}{\vspace{0.3cm}} \nc{\tk}{^{(k)}} \nc{\tp}{^\prime}
\nc{\ttp}{^{\prime\prime}} \nc{\svp}{\vspace{2cm}}
\nc{\vp}{\vspace{8cm}} \nc{\proofbegin}{\noindent{\bf Proof: }}
\nc{\proofend}{$\blacksquare$ \vspace{0.3cm}}
\nc{\modg}[1]{\!<\!\!{#1}\!\!>}
\nc{\intg}[1]{F_C(#1)} \nc{\lmodg}{\!
<\!\!} \nc{\rmodg}{\!\!>\!}
\nc{\cpi}{\widehat{\Pi}}
\nc{\sha}{{\mbox{\cyr X}}}  

\newfont{\scyr}{wncyr10 scaled 550}
\nc{\ssha}{\mbox{\bf \scyr X}}
\nc{\shap}{{\mbox{\cyrs X}}} 
\nc{\shpr}{\diamond}    
\nc{\shp}{\ast} \nc{\shplus}{\shpr^+}
\nc{\shprc}{\shpr_c}    
\nc{\msh}{\ast} \nc{\zprod}{m_0} \nc{\oprod}{m_1}
\nc{\vep}{\epsilon} \nc{\labs}{\mid\!} \nc{\rabs}{\!\mid}
\nc{\sqmon}[1]{\langle #1\rangle}

\nc{\mmbox}[1]{\mbox{\ #1\ }} \nc{\dep}{\mrm{dep}} \nc{\fp}{\mrm{FP}}
\nc{\rchar}{\mrm{char}} \nc{\End}{\mrm{End}} \nc{\Fil}{\mrm{Fil}}
\nc{\Mor}{Mor\xspace} \nc{\gmzvs}{gMZV\xspace}
\nc{\gmzv}{gMZV\xspace} \nc{\mzv}{MZV\xspace}
\nc{\mzvs}{MZVs\xspace} \nc{\Hom}{\mrm{Hom}} \nc{\id}{\mrm{id}}
\nc{\im}{\mrm{im}} \nc{\incl}{\mrm{incl}} \nc{\map}{\mrm{Map}}
\nc{\mchar}{\rm char} \nc{\nz}{\rm NZ} \nc{\supp}{\mathrm Supp}

\nc{\Alg}{\mathbf{Alg}} \nc{\Bax}{\mathbf{Bax}} \nc{\bff}{\mathbf f}
\nc{\bfk}{{\bf k}} \nc{\bfone}{{\bf 1}} \nc{\bfx}{\mathbf x}
\nc{\bfy}{\mathbf y}
\nc{\base}[1]{\bfone^{\otimes ({#1}+1)}} 
\nc{\Cat}{\mathbf{Cat}}

\nc{\detail}{\marginpar{\bf More detail}
    \noindent{\bf Need more detail!}
    \svp}
\nc{\Int}{\mathbf{Int}} \nc{\Mon}{\mathbf{Mon}}
\nc{\rbtm}{{shuffle }} \nc{\rbto}{{Rota-Baxter }}
\nc{\remarks}{\noindent{\bf Remarks: }} \nc{\Rings}{\mathbf{Rings}}
\nc{\Sets}{\mathbf{Sets}} \nc{\wtot}{\widetilde{\odot}}
\nc{\wast}{\widetilde{\ast}} \nc{\bodot}{\bar{\odot}}
\nc{\bast}{\bar{\ast}} \nc{\hodot}[1]{\odot^{#1}}
\nc{\hast}[1]{\ast^{#1}} \nc{\mal}{\mathcal{O}}
\nc{\tet}{\tilde{\ast}} \nc{\teot}{\tilde{\odot}}
\nc{\oex}{\overline{x}} \nc{\oey}{\overline{y}}
\nc{\oez}{\overline{z}} \nc{\oef}{\overline{f}}
\nc{\oea}{\overline{a}} \nc{\oeb}{\overline{b}}
\nc{\weast}[1]{\widetilde{\ast}^{#1}}
\nc{\weodot}[1]{\widetilde{\odot}^{#1}} \nc{\hstar}[1]{\star^{#1}}
\nc{\lae}{\langle} \nc{\rae}{\rangle}
\nc{\lf}{\lfloor}
\nc{\rf}{\rfloor}

\nc{\QQ}{{\mathbb Q}}
\nc{\RR}{{\mathbb R}} \nc{\ZZ}{{\mathbb Z}}

\nc{\cala}{{\mathcal A}} \nc{\calb}{{\mathcal B}}
\nc{\calc}{{\mathcal C}}
\nc{\cald}{{\mathcal D}} \nc{\cale}{{\mathcal E}}
\nc{\calf}{{\mathcal F}} \nc{\calg}{{\mathcal G}}
\nc{\calh}{{\mathcal H}} \nc{\cali}{{\mathcal I}}
\nc{\call}{{\mathcal L}} \nc{\calm}{{\mathcal M}}
\nc{\caln}{{\mathcal N}} \nc{\calo}{{\mathcal O}}
\nc{\calp}{{\mathcal P}} \nc{\calr}{{\mathcal R}}
\nc{\cals}{{\mathcal S}} \nc{\calt}{{\mathcal T}}
\nc{\calu}{{\mathcal U}} \nc{\calw}{{\mathcal W}} \nc{\calk}{{\mathcal K}}
\nc{\calx}{{\mathcal X}} \nc{\CA}{\mathcal{A}}

\nc{\fraka}{{\mathfrak a}} \nc{\frakA}{{\mathfrak A}}
\nc{\frakb}{{\mathfrak b}} \nc{\frakB}{{\mathfrak B}}
\nc{\frakc}{{\mathfrak c}}
\nc{\frakD}{{\mathfrak D}} \nc{\frakF}{\mathfrak{F}}
\nc{\frakf}{{\mathfrak f}} \nc{\frakg}{{\mathfrak g}}
\nc{\frakH}{{\mathfrak H}} \nc{\frakL}{{\mathfrak L}}
\nc{\frakM}{{\mathfrak M}} \nc{\bfrakM}{\overline{\frakM}}
\nc{\frakm}{{\mathfrak m}} \nc{\frakP}{{\mathfrak P}}
\nc{\frakN}{{\mathfrak N}} \nc{\frakp}{{\mathfrak p}}
\nc{\frakS}{{\mathfrak S}} \nc{\frakT}{\mathfrak{T}}
\nc{\frakX}{{\mathfrak X}}
\nc{\BS}{\mathbb{S
}}

\font\cyr=wncyr10 \font\cyrs=wncyr7

\nc{\ID}{{\rm I}}\nc{\lbar}[1]{\overline{#1}}\nc{\bre}{{\rm bre}}
\nc{\sd}{\cals}\nc{\rb}{\rm RB}\nc{\A}{\rm A}\nc{\LL}{\rm L}\nc{\tx}{\tilde{X}}
\nc{\col}{\Delta_{RT}}\nc{\mul}{m_{RT}}\nc{\ul}{u_{RT}}\nc{\epl}{\varepsilon_{RT}}
\nc{\hl}{H_{RT}}\nc{\arro}[1]{#1}\nc{\px}{P_{\tx}}\nc{\pw}{P_{\mathfrak{w}}}\nc{\pl}{B_\omega^+}
\nc{\pp}{\pl}\nc{\ppp}[1]{B^+(#1)}\nc{\dw}{\diamond_{\mathfrak{w}}}\nc{\dl}{\diamond_{\rm \ell}}
\nc{\ncshaw}{\sha^{{\rm NC}}_{\Omega}}\nc{\ncshal}{\sha^{{\rm NC}}_{{\rm RT}}}
\nc{\ver}{\rm V}\nc{\ld}{l}\nc{\del}{\Delta_{{\rm \ell}}}\nc{\epsl}{\epsilon_{{\rm \ell}}}
\nc{\uul}{u_{{\rm \ell}}}\nc{\oneh}{\mathbf{1}}\nc{\onew}{\mathbf{1}}
\nc{\etree}{1} \nc{\conc}{m_{RT}}
\nc{\hrtb}{H_{RT}(X\sqcup\Omega)} \nc{\hrts}{H_{RT}(X, \Omega)}\nc{\rts}{\mathcal{T}(X, \Omega)}\nc{\rfs}{\mathcal{F}(X, \Omega)} \nc{\ncshall}{\sha^{{\rm NC}}_{{\rm RT}}} \nc{\ldl}{\leq_{\mathrm{dl}}} \nc{\pla}{B_{\alpha}^{+}} \nc{\plb}{B_{\beta}^{+}}

\nc{\bim}[1]{#1}  \nc{\shaop}{\sha_{\Omega}^{+}}  \nc{\shao}{\sha_{\Omega}}
\nc{\bbim}[2]{#1 #2} \nc{\bbbim}[2]{#1,\, #2} \nc{\RBF}{{\rm RBF}}
\nc{\frbf}{F_{\RBF}} \nc{\shaf}{\ssha_{\tiny{\Omega}}} \nc{\sham}{\diamond_{\Omega}}

\nc{\dnx}{\Delta_n A} \nc{\dx}{\Delta A} \nc{\dgp}{{\rm deg_{P}}}
\nc{\dgt}{{\rm deg_{T}}} \nc{\dg}{{\rm deg}} \nc{\ida}{ID($A$)} \nc{\tu}{\tilde{u}} \nc{\tv}{\tilde{v}}
 \nc{\fbase}{\calb} \nc{\LF}{\mathrm{RF}} \nc{\FFA}{\mathrm{LF}} \nc{\irr}{\mathrm{Irr}}
 \nc{\result}{\bfk\mathrm{Irr}(S_n)}  \nc{\I}{I_{\mathrm{ID},n}^0}
 \nc{\nrs}{\calr_n^\star} \nc{\ii}{\mathrm{I}} \nc{\iii}{\mathrm{II}}
\nc{\intl}{{\rm int}}\nc{\ws}[1]{{#1}}\nc{\deleted}[1]{\delete{#1}}\nc{\plas}{placements\xspace}

\nc{\Id}{\mathrm{Id}} \nc{\Irr}{\mathrm{Irr}}

\nc{\tos}{totally ordered set } \nc{\nes}{nonempty set}
\nc{\baa}{\succ_\alpha}
\nc{\laa}{\prec_\alpha}
\nc{\bba}{\succ_\beta}
\nc{\lba}{\prec_\beta}
\nc{\bdt} {\succ_\delta}
\nc{\lia} {\prec_\iota}
\nc{\bka}{\succ_\kappa}
\nc{\lmu}{\prec_\mu}
\nc{\Po}{(P_\omega)_{\omega\in \Omega}}

\nc{\apr}{\cdot} \nc{\spr}{\cdot}

\usetikzlibrary{calc}
\newlength\xch\newlength\dbj
\newif\ifqdd
\usepackage{mathtools,enumitem}
\allowdisplaybreaks[1]

\renewcommand*{\backref}[1]{}

\renewcommand*{\backrefalt}[4]{%
  \ifcase #1\relax
  \else #2%
  \fi}

\begin{document}

\title[Rota-Baxter operators on Lie affgebras]{Rota-Baxter Operators and Factorization of Lie Affgebras}

\author{JiaJia Han}
\address{School of Mathematics and Statistics, Nanjing University of Information Science and Technology, Nanjing, Jiangsu 210044, P. R. China}
\email{jjhan@nuist.edu.cn}

\author{Abror Khudoyberdiyev}
\address{V.I.Romanovskiy Institute of Mathematics Uzbekistan Academy of Sciences, Tashkent, Uzbekistan. National University of Uzbekistan, Tashkent, Uzbekistan}
\email{khabror@mail.ru}

\author{Yunzhou Xie}\address{Department of Mathematics, Imperial College London, SW7 2AZ, London, UK}
\email{yx3021@ic.ac.uk}

\author{Yi Zhang}\address{School of Mathematics and Statistics; Center for Applied Mathematics of Jiangsu Province/Jiangsu
International Joint Laboratory on System Modeling and Data Analysis, Nanjing University of Information Science and Technology, Nanjing, Jiangsu, 210044, P.\,R. China}
\email{zhangy2016@nuist.edu.cn}

\date{\today}

\begin{abstract}
We introduce Rota-Baxter operators on Lie affgebras by an identity independent of the choice of origin and prove that they induce descendant Lie affgebra structures. We establish a bijection between ordered affine Lie factorizations and idempotent operators split at a chosen idempotent fixed point. A cokernel obstruction determines the existence of this point, while splitting at such a point is equivalent to commutation of the linear part with the two structural endomorphisms. We also determine when weighted descendants satisfy the affine Jacobi identity and when operators of nonzero weight give affine doubles. For canonical Lie affgebras, we describe all affine Rota-Baxter operators by Lie subalgebra decompositions and translation vectors, and classify them up to automorphism in dimensions at most two.
\end{abstract}

\subjclass[2020]{17B05, 17B38, 17B40}

\keywords{Lie affgebra, Rota-Baxter operator, descendant bracket, fixed point, affine factorization, affine double, classification in low dimensions}

\maketitle

\setcounter{tocdepth}{2}

\vspace{-1cm}

\tableofcontents

\vspace{-1cm}

\section{Introduction}\label{s1}

The purpose of this paper is to introduce Rota-Baxter operators on Lie affgebras and to determine how they produce descendant brackets and ordered affine Lie factorizations. The operator identity is formulated independently of an origin. For idempotent operators, a factorization requires an idempotent fixed point whose fibre is closed under the original bracket.

\subsection{Rota-Baxter operators and Lie algebra decompositions}\label{u1}

Throughout, $\mathbf{k}$ is a field of characteristic zero. Rota-Baxter operators on Lie algebras occur in the operator approach to classical Yang-Baxter equations and Lie algebra factorization~\cite{G2012,STS1983}. For a Lie algebra $(\mathfrak g,[-,-])$, a linear map $R:\mathfrak g\to\mathfrak g$ is a Rota-Baxter operator of weight $-1$ if
$$
[R(x),R(y)]=R\big([R(x),y]+[x,R(y)]-[x,y]\big),
\quad x,y\in\mathfrak g.
$$
The expression inside $R$, denoted by $[x,y]_{R,-1}$, defines the descendant Lie bracket, and $R$ is a homomorphism from the descendant Lie algebra to the original one~\cite{BGN2010}.

At weight zero, Rota-Baxter operators are $\mathcal O$-operators for the adjoint representation, while an arbitrary representation may be used in the definition of an $\mathcal O$-operator~\cite{K1999}. Their relations with tensor solutions of Yang-Baxter equations and post-Lie structures are established in~\cite{BGN2010,BGN2011}, and their deformation theory is developed in~\cite{TBGS2019}. Factorization and differentiation for Rota-Baxter Lie groups are treated in~\cite{GLS2021}, and post-groups connect these constructions with post-Lie algebras and solutions of the Yang-Baxter equation on sets~\cite{BGST2024}. Doubles and matched pairs relate quadratic Rota-Baxter Lie algebras to factorizable Lie bialgebras~\cite{LS2023}, while global integration requires the conditions discussed in~\cite{JSZ2025}.

For the factorization problem considered here, the relevant linear result concerns projections. If $R^2=R$, the Rota-Baxter identity of weight $-1$ is equivalent to closure of both $\operatorname{Im}R$ and $\ker R$ under the Lie bracket. Conversely, an ordered vector space decomposition $\mathfrak g=\mathfrak g_+\oplus\mathfrak g_-$ into Lie subalgebras determines such an operator by projection onto $\mathfrak g_+$~\cite{G2012}. Thus idempotent Rota-Baxter operators give an algebraic description of ordered Lie subalgebra decompositions.

\subsection{The affine factorization problem}\label{u2}

The passage from vector spaces to affine spaces is motivated in part by mechanics formulated without a preferred frame. Lie structures on affine bundles and their relations with affine Poisson and Jacobi structures were studied in~\cite{GGU2003,GGU2004,GGU2005,GGU2007,T1985}. For brackets taking values in the affine space itself, Brzezi\'nski and Papworth gave an algebraic formulation of Lie affgebras using heaps~\cite{BP2023a}. We use the homogeneous form of the affine Jacobi identity introduced in~\cite{ABR2025}. A Lie affgebra is an affine space with a biaffine bracket satisfying the corresponding affine antisymmetry and Jacobi identities. These identities are expressed through the abelian heap operation $\langle x,y,z\rangle$, represented by $x-y+z$ in any choice of affine coordinates.

The relation between heaps, modules and affine spaces is developed in~\cite{BBRS2024,BBRS2025}. Related algebraic constructions include trusses~\cite{B2019,B2020}, affine Leibniz and Poisson algebras~\cite{BRP2026a,BRP2026b}, affine Nijenhuis operators and Rota-Baxter operators on trusses~\cite{BP2023b,CEMM2025}. For Lie affgebras, Proposition~\ref{tstr}, recalled from~\cite{ABR2025}, describes the bracket at a chosen origin by a tangent Lie algebra, two linear endomorphisms satisfying a generalized derivation identity, and a constant term. This connects their structure and classification with the theory of generalized derivations~\cite{LL2000,BK2025}.

Extending the projection correspondence to this setting requires more than imposing a Rota-Baxter identity on the tangent Lie algebra. At an origin $o$, an affine map has the form $P(x)=R(x)+p$, where $R$ is its linear part and $p=P(o)-o$ is its translation vector. The ordinary identity for $R$ does not constrain $p$ or the linear and constant terms of the affine bracket. The first question is therefore which identity for $P$ includes these terms and remains independent of the origin.

At weight $-1$, the three coefficients in the descendant sum to one, so the heap operation gives
$$
\{x,y\}_P
=\big\langle\{P(x),y\},\{x,y\},\{x,P(y)\}\big\rangle.
$$
We ask whether $P(\{x,y\}_P)=\{P(x),P(y)\}$ already implies the affine Jacobi identity for this operation. If it does, the next question is when an idempotent $P$ decomposes the original Lie affgebra into its image and a fibre.

An affine decomposition requires a common point. We write $\mathfrak a=\mathfrak a_+\oplus_e\mathfrak a_-$ for an ordered affine Lie factorization when the two factors are Lie subaffgebras with intersection $\{e\}$ and every point has a unique expression $\langle u,e,v\rangle$, with $u\in\mathfrak a_+$ and $v\in\mathfrak a_-$. For the factors $\operatorname{Im}P$ and $P^{-1}(e)$, this point must satisfy $\{e,e\}=P(e)=e$. Such a point need not exist. If it exists, the operator identity ensures that the fibre is closed under the descendant bracket, but closure under the original bracket is a separate condition. The problem is therefore to determine both when an idempotent fixed point exists and when its fibre is a Lie subaffgebra of the original Lie affgebra.

\subsection{Main results}\label{u3}

The main result extends the correspondence between idempotent Rota-Baxter operators and Lie subalgebra decompositions to Lie affgebras. Its first step is the descendant theorem: at weight $-1$, the affine operator identity implies the affine Jacobi identity, without an additional compatibility assumption.

\begin{theorem}[= Theorem~\ref{t1}]
Let $(\mathfrak a,\{-,-\})$ be a Lie affgebra and let $P:\mathfrak a\to\mathfrak a$ be an affine map such that $P(\{x,y\}_P)=\{P(x),P(y)\}$ for all $x,y\in\mathfrak a$. Then $\{-,-\}_P$ is a Lie affgebra bracket, and $P$ is a homomorphism from the descendant Lie affgebra to the original one. Its linear part is a Rota-Baxter operator of weight $-1$ on every tangent Lie algebra.
\end{theorem}

We call a map satisfying this identity an affine Rota-Baxter operator. The theorem requires neither idempotency nor a fixed point. These conditions enter in the factorization problem. If $P$ is idempotent, its image is a Lie subaffgebra, but a fibre over an idempotent fixed point need not be. We call $P$ split at such a point $e$ when $P^{-1}(e)$ is a Lie subaffgebra.

\begin{theorem}[= Theorem~\ref{t6}]
Ordered affine Lie factorizations of a Lie affgebra $\mathfrak a$ are in bijection with pairs $(P,e)$ consisting of an idempotent affine Rota-Baxter operator and an idempotent fixed point at which it is split. The factors are $\operatorname{Im}P$ and $P^{-1}(e)$, and $P$ is the projection onto the first factor along the second. The complementary map $Q_e(x)=\langle x,P(x),e\rangle$ is an idempotent affine Rota-Baxter operator corresponding to the reversed factorization.
\end{theorem}

The additional affine conditions in this correspondence are explicit. Theorem~\ref{t4} gives a cokernel class whose vanishing is equivalent to the existence of an idempotent fixed point and describes the affine space of all such points. After choosing one of them as origin, Theorem~\ref{t5} identifies closure of the fibre with commutation of the projection and the two structural endomorphisms. These two conditions determine when a tangent Lie algebra decomposition extends to an affine Lie factorization.

For a Lie algebra $(\mathfrak g,[-,-])$, the bracket $\{x,y\}=[x,y]+y$ is a special case of~\cite[Example~2.7]{ABR2025}. We call this the \textbf{canonical Lie affgebra} associated with $\mathfrak g$ and denote it by $\mathfrak a_{\mathrm{can}}(\mathfrak g)$. On these Lie affgebras, the affine Rota-Baxter identity itself forces idempotency and gives the following description of all operators.

\begin{theorem}[= Theorem~\ref{t9}]
The affine Rota-Baxter operators on the canonical Lie affgebra associated with $\mathfrak g$ are exactly the maps
$$
P(x_++x_-)=x_++p,
\quad x_+\in\mathfrak g_+,\quad x_-\in\mathfrak g_-,
$$
where $\mathfrak g=\mathfrak g_+\oplus\mathfrak g_-$ is a vector space decomposition into Lie subalgebras and $p\in\mathfrak g_-$ satisfies $[p,\mathfrak g_+]\subseteq\mathfrak g_+$. Every such operator is idempotent and split at $p$.
\end{theorem}

The theorem describes each operator by an ordered Lie subalgebra decomposition and a vector satisfying the stated conditions. Classification also requires the affine automorphisms acting on these objects. Every automorphism of a canonical Lie affgebra is a Lie algebra automorphism followed by a translation by a central vector (Proposition~\ref{p14}). Both transformations enter the equivalence of operators. For an abelian tangent Lie algebra, central translations remove the translation vector, and the operators are classified by rank in every finite dimension (Theorem~\ref{t10}). For the nonabelian Lie algebra of dimension two, the equivalence classes are represented by five individual operators and two sets indexed by $\alpha\in\mathbf{k}^{\times}$ and $t\in\mathbf{k}$, respectively (Theorem~\ref{t11}). Different rows and different parameter values give inequivalent operators. These results classify affine Rota-Baxter operators on canonical Lie affgebras of dimension at most two.

The proofs use the following properties of the affine bracket and the operator.
\begin{enumerate}[label={\rm(\arabic*)}]
\item Comparing the bilinear, linear and constant parts of the affine identity gives the coordinate equations in Corollary~\ref{c2}. The identity $[x,y]_{R,-1}=[R(x),R(y)]-[x-R(x),y-R(y)]$ then reduces the descendant calculation to the original structural identity and the derivation property of $x\mapsto[p,x]$ (Theorem~\ref{t1}).
\item A change of origin alters both the constant term of the bracket and the translation vector of the operator by values of one linear map. Solving the resulting linear system determines the common points (Theorem~\ref{t4}). At any such point, both structural endomorphisms preserve the image of the projection. Closure of the fibre under the original bracket is equivalent to their preservation of its kernel, or equivalently to their commutation with the projection (Theorem~\ref{t5}). These conditions give the factorization correspondence in Theorem~\ref{t6}.
\item For canonical Lie affgebras, the coordinate equations force the linear part to be a projection and the translation vector to lie in its kernel and normalize its image (Theorem~\ref{t9}). Lie algebra automorphisms act on the ordered image and kernel, while central translations also change the translation vector as specified in Proposition~\ref{p14}. That proposition gives the equivalence relations used in Theorems~\ref{t10} and~\ref{t11}.
\end{enumerate}

We also determine what changes at other weights. The three coefficients in the ordinary weighted descendant do not sum to one when the weight differs from $-1$. Subtracting equal multiples of the two diagonal values gives an operation independent of the origin, with the ordinary weighted descendant as its bilinear part. Under the affine operator identity, this operation satisfies the affine Jacobi identity precisely when the structural condition in Theorem~\ref{t3} holds. Example~\ref{x1} shows that this condition is needed at weight $1$. At nonzero weight, Theorem~\ref{t7} determines when the affine graph is closed under the componentwise bracket. In that case, the graph and diagonal give an affine double whose tangent Lie algebra is the Rota-Baxter double of~\cite{LS2023} (Theorem~\ref{t8}). The bracket transported from the graph and the weighted descendant generally differ. At an idempotent fixed point of an idempotent affine Rota-Baxter operator, graph closure at weight $-1$ is equivalent to splitting (Proposition~\ref{p12}).

\noindent\textbf{Organization of the paper.} Section~\ref{s2} recalls the structure of Lie affgebras and introduces affine Rota-Baxter operators. It proves the descendant theorem, determines when an idempotent fixed point exists and then treats arbitrary weights. Section~\ref{s3} studies the image and affine kernel, proves the splitting condition and establishes the correspondence with ordered affine Lie factorizations. It also constructs affine doubles at nonzero weight and identifies their tangent Lie algebras. Section~\ref{s4} applies the factorization results to Lie affgebras with $\kappa=\lambda=\gamma\operatorname{id}$, for a fixed Lie algebra $\mathfrak g$ and parameters $\gamma\in\mathbf{k}$ and $s\in\mathfrak g$. It then describes all affine Rota-Baxter operators on canonical Lie affgebras and gives their classification up to automorphism in dimensions at most two.

\noindent\textbf{Notation.} All vector and affine spaces are over the field $\mathbf{k}$ fixed above, and $\mathbf{k}^{\times}=\mathbf{k}\setminus\{0\}$. For a vector space $V$, we write $\operatorname{End}(V)=\operatorname{End}_{\mathbf{k}}(V)$ for its space of linear endomorphisms and $\operatorname{id}$ for the identity map. Products of maps denote composition, with $ST=S\circ T$. For linear endomorphisms, $[S,T]=ST-TS$. For a Lie algebra $\mathfrak g$, we write $\operatorname{ad}_a(x)=[a,x]$ for $a,x\in\mathfrak g$.

\section{Lie affgebras and Rota-Baxter operators}\label{s2}

We recall the structure of Lie affgebras and their subaffgebras, then introduce affine Rota-Baxter operators through an identity independent of the choice of origin. The main result is the descendant theorem. We next describe the obstruction to an idempotent fixed point and then determine the structural identity required at other weights.

\subsection{Affine spaces, affine maps and tangent spaces}\label{u5}

We recall the definitions of affine spaces and affine maps used in~\cite{ABR2025,BP2023a}.

A nonempty set $X$ with a ternary operation $\langle-,-,-\rangle:X\times X\times X\to X$ is an \textbf{abelian heap} if, for all $x,y,z,u,v\in X$,
$$
\langle x,y,z\rangle=\langle z,y,x\rangle,\quad
\langle x,x,z\rangle=z,
\quad
\langle\langle x,y,z\rangle,u,v\rangle
=\langle x,y,\langle z,u,v\rangle\rangle.
$$
Let $(X,\langle-,-,-\rangle_X)$ and $(Y,\langle-,-,-\rangle_Y)$ be abelian heaps. A map $f:X\to Y$ is a \textbf{homomorphism of heaps} if
$$
f(\langle x,y,z\rangle_X)=\langle f(x),f(y),f(z)\rangle_Y
$$
for all $x,y,z\in X$~\cite{BBRS2024,B2019}. For five arguments we write
$
\langle a,b,c,d,e\rangle=\langle\langle a,b,c\rangle,d,e\rangle.
$
The heap identity makes this independent of the parenthesization. Choosing $e\in X$ gives an abelian group with addition $x+_e y=\langle x,e,y\rangle$ and zero $e$. Conversely, every abelian group defines an abelian heap by $\langle x,y,z\rangle=x-y+z$.

\begin{defn}\label{d0}
\textup{\cite{ABR2025,OS1966}} An \textbf{affine space} over $\mathbf{k}$ is a nonempty set $\mathfrak a$ equipped with ternary operations
$$
\langle-,-,-\rangle:\mathfrak a^3\to\mathfrak a,
\quad
(-,-,-):\mathbf{k}\times\mathfrak a\times\mathfrak a\to\mathfrak a
$$
satisfying the following conditions for all $\alpha,\beta\in\mathbf{k}$ and $x,y,z\in\mathfrak a$:
\begin{enumerate}[label=\textup{(\alph*)},leftmargin=*]
\item $(\mathfrak a,\langle-,-,-\rangle)$ is an abelian heap.
\item The maps $(\alpha,x,-):\mathfrak a\to\mathfrak a$ and $(-,x,y):\mathbf{k}\to\mathfrak a$ are homomorphisms of heaps, where the heap operation on $\mathbf{k}$ is $\langle\alpha,\beta,\gamma\rangle=\alpha-\beta+\gamma$ for $\alpha,\beta,\gamma\in\mathbf{k}$.
\item For each $x$, the map $(-,x,-)$ is an action of the multiplicative monoid of $\mathbf{k}$, namely
$$
(1,x,y)=y,\quad (\alpha,x,(\beta,x,y))=(\alpha\beta,x,y).
$$
\item $(0,x,y)=x$.
\item $(\alpha,x,y)=\langle(\alpha,z,y),(\alpha,z,x),x\rangle$.
\end{enumerate}
For a fixed $x\in\mathfrak a$, we call $(-,x,-)$ the \textbf{affine action} with base $x$.
\end{defn}

A map $f:\mathfrak a\to\mathfrak b$ between affine spaces is \textbf{affine} if it preserves both operations, that is,
$$
f(\langle x,y,z\rangle)=\langle f(x),f(y),f(z)\rangle,
\quad
f((\alpha,x,y))=(\alpha,f(x),f(y))
$$
for all $x,y,z\in\mathfrak a$ and $\alpha\in\mathbf{k}$~\cite{ABR2025}. A map of two affine variables is \textbf{biaffine} if it is affine in each variable separately. An affine map $P:\mathfrak a\to\mathfrak a$ is \textbf{idempotent} if $P^2=P$.

For $e\in\mathfrak a$, the operations
$$
x+_e y=\langle x,e,y\rangle,\quad
\alpha\cdot_e x=(\alpha,e,x)
$$
make the same underlying set a vector space with zero $e$. Following~\cite{ABR2025}, we denote this space by $T_e\mathfrak a$ and call it the \textbf{tangent space} to $\mathfrak a$ at $e$. In terms of these vector space operations, written without subscripts, we have
$$
\langle x,y,z\rangle=x-y+z,\quad
(\alpha,x,y)=(1-\alpha)x+\alpha y.
$$
For $e,u\in\mathfrak a$, the map $\tau_e^u:T_e\mathfrak a\to T_u\mathfrak a$, $\tau_e^u(x)=\langle x,e,u\rangle$, is a vector space isomorphism~\cite{BP2023a}. An affine combination $\alpha_1x_1+\cdots+\alpha_nx_n$ is independent of the chosen origin whenever $\alpha_1+\cdots+\alpha_n=1$.

Definition~\ref{d0} is equivalent to the usual description of an affine space by a free and transitive action of the additive group of a vector space $V$, called its \textbf{translation space}~\cite{BBRS2024,BP2023a}. Indeed, $T_e\mathfrak a$ acts on $\mathfrak a$ by $(x,v)\mapsto\langle x,e,v\rangle$. Conversely, given such an action, write $x+v$ for the action of $v\in V$ and $y-x\in V$ for the unique vector taking $x$ to $y$. Definition~\ref{d0} then holds with
$$
\langle x,y,z\rangle=x+(z-y),\quad
(\alpha,x,y)=x+\alpha(y-x).
$$

If $\mathfrak a$ and $\mathfrak b$ have translation spaces $V$ and $W$, an affine map $f:\mathfrak a\to\mathfrak b$ determines a unique linear map $\vec f:V\to W$ with $f(x+v)=f(x)+\vec f(v)$. We call $\vec f$ the \textbf{linear part} of $f$. For each $e\in\mathfrak a$, the map $f:T_e\mathfrak a\to T_{f(e)}\mathfrak b$ is linear and corresponds to $\vec f$ under the identifications with $V$ and $W$. Thus $f(x)=f(e)+\vec f(x-e)$~\cite{ABR2025,BP2023a}. A nonempty subset $\mathfrak b\subseteq\mathfrak a$ is an \textbf{affine subspace} if, for one (and hence every) $a\in\mathfrak b$, there is a vector subspace $W\subseteq V$ such that $\mathfrak b=a+W$.

\begin{defn}
\textup{\cite{ABR2025}} A \textbf{Lie affgebra} is an affine space $\mathfrak a$ equipped with a biaffine map $\{-,-\}:\mathfrak a\times\mathfrak a\to\mathfrak a$ such that, for all $x,y,z\in\mathfrak a$,
\begin{align*}
&\langle\{x,y\},\{x,x\},\{y,x\}\rangle=\{y,y\},\\
&\big\langle
\{x,\{y,z\}\},
\{x,\{x,x\}\},
\{y,\{z,x\}\},
\{y,\{y,y\}\},
\{z,\{x,y\}\}
\big\rangle
=
\{z,\{z,z\}\}.
\end{align*}
An affine map preserving the bracket is called a \textbf{homomorphism} of Lie affgebras. A bijective homomorphism is an \textbf{isomorphism}, and an isomorphism from a Lie affgebra to itself is an \textbf{automorphism}. A point $e\in\mathfrak a$ is \textbf{idempotent} if $\{e,e\}=e$.
\end{defn}

\begin{prop}\label{ttan}
\textup{\cite[Theorem~2.6]{ABR2025}} Let $(\mathfrak a,\{-,-\})$ be a Lie affgebra. For every $o\in\mathfrak a$, the vector space $T_o\mathfrak a$ is a Lie algebra with bracket
$$
[x,y]_o
=
\{x,y\}-\{x,o\}+\{o,o\}-\{o,y\}.
$$
The expression is evaluated in $T_o\mathfrak a$. Tangent Lie algebras at different points are isomorphic.
\end{prop}

We recall the following description of Lie affgebra brackets.

\begin{prop}\label{tstr}
\textup{\cite[Theorem~3.1]{ABR2025}} Let $(\mathfrak g,[-,-])$ be a Lie algebra and let $\kappa,\lambda\in\operatorname{End}(\mathfrak g)$ satisfy
\begin{equation}\label{e1}
\lambda([x,y])
=
[\lambda(x),y]-[x,\kappa(y)]+[x,\lambda(y)]
\end{equation}
for all $x,y\in\mathfrak g$. Then, for every $s\in\mathfrak g$,
\begin{equation}\label{e2}
\{x,y\}
=
[x,y]+\kappa(x)+\lambda(y-x)+s
\end{equation}
defines a Lie affgebra on the affine space underlying $\mathfrak g$.
Conversely, after choosing any point, every Lie affgebra is of this form.
\end{prop}

We denote the Lie affgebra in Eq.~(\ref{e2}) by $\mathfrak a(\mathfrak g;\kappa,\lambda,s)$. We call $\kappa$ and $\lambda$ its \textbf{structural endomorphisms} at the chosen origin, and $s$ its constant term. They can be recovered from the bracket by $s=\{0,0\}$, $\kappa(x)=\{x,x\}-s$ and $\lambda(x)=\{0,x\}-s$.

In these coordinates, the chosen origin is $0\in\mathfrak g$. For an affine map $P:\mathfrak a\to\mathfrak a$, set
$$
p=P(0)\in\mathfrak g,\quad R(x)=P(x)-P(0)\quad(x\in\mathfrak g).
$$
Since $P$ preserves the heap operation and the affine action, we have
\begin{align*}
R(x+y)&=P(x+y)-P(0)=P(x)+P(y)-2P(0)=R(x)+R(y),\\
R(\alpha x)&=P(\alpha x)-P(0)=\alpha\bigl(P(x)-P(0)\bigr)=\alpha R(x)
\end{align*}
for all $x,y\in\mathfrak g$ and $\alpha\in\mathbf{k}$. Thus $R\in\operatorname{End}(\mathfrak g)$ is the linear part of $P$, and $p$ is its translation vector relative to the chosen origin. Consequently, $P(x)=R(x)+p$. This decomposition is unique: evaluating at $0$ determines $p$, and then $R(x)=P(x)-p$. We also write $P=R+p$.

The \textbf{quasicentroid} is $QC(\mathfrak g):=\{D\in\operatorname{End}(\mathfrak g)\mid[D(x),y]=[x,D(y)]\text{ for all }x,y\in\mathfrak g\}$. Eq.~(\ref{e1}) implies $\kappa\in QC(\mathfrak g)$~\cite{ABR2025,LL2000}. A direct calculation gives
$$
[x,(\lambda-\kappa)(y)]-[y,(\lambda-\kappa)(x)]-(\lambda-\kappa)([x,y]) =\kappa([x,y])-[\kappa(x),y].
$$
The \textbf{centroid} is $C(\mathfrak g):=\{D\in\operatorname{End}(\mathfrak g)\mid D([x,y])=[D(x),y]=[x,D(y)]\text{ for all }x,y\in\mathfrak g\}$. Consequently, $\lambda-\kappa$ is a derivation exactly when $\kappa\in C(\mathfrak g)$~\cite{ABR2025}.

We also recall the description of Lie affgebra homomorphisms.

\begin{prop}\label{thom}
\textup{\cite[Theorem~3.4]{ABR2025}} Let $\mathfrak a=\mathfrak a(\mathfrak g;\kappa,\lambda,s)$ and $\mathfrak a'=\mathfrak a(\mathfrak g';\kappa',\lambda',s')$.
Let $\psi:\mathfrak g\to\mathfrak g^{\prime}$ be linear and $q\in\mathfrak g^{\prime}$. The affine map $\varphi(x)=\psi(x)+q$ is a Lie affgebra homomorphism if and only if $\psi:\mathfrak g\to\mathfrak g'$ is a Lie algebra homomorphism and
\begin{equation}\label{e3}
\psi\kappa=\kappa'\psi,\quad
\psi\lambda=(\operatorname{ad}_q+\lambda')\psi,\quad
\psi(s)=s'-q+\kappa'(q).
\end{equation}
\end{prop}

\subsection{Lie subaffgebras}\label{u6}

\begin{defn}
\textup{\cite{ABR2025}} Let $(\mathfrak a,\{-,-\})$ be a Lie affgebra. A nonempty affine subspace $\mathfrak b\subseteq\mathfrak a$ is called a \textbf{Lie subaffgebra} if $\{\mathfrak b,\mathfrak b\}\subseteq\mathfrak b$.
\end{defn}

We shall use the following characterization of Lie subaffgebras.

\begin{prop}\label{p1}
\textup{\cite[Proposition~3.8]{ABR2025}} Let $\mathfrak a=\mathfrak a(\mathfrak g;\kappa,\lambda,s)$ and let $\mathfrak b=a+\mathfrak h$ be an affine subspace. Then $\mathfrak b$ is a Lie subaffgebra if and only if $\mathfrak h$ is a Lie subalgebra of $\mathfrak g$ and
$$
\kappa(a)-a+s\in\mathfrak h,\quad
\kappa(\mathfrak h)\subseteq\mathfrak h,\quad
(\lambda+\operatorname{ad}_a)(\mathfrak h)\subseteq\mathfrak h.
$$
In that case
$$
\mathfrak b
\cong
\mathfrak a
\big(
\mathfrak h;
\kappa|_{\mathfrak h},
(\lambda+\operatorname{ad}_a)|_{\mathfrak h},
\kappa(a)-a+s
\big).
$$
\end{prop}

The special case in which $a$ is chosen as the origin will be used frequently.

\begin{coro}\label{c1}
Assume $\mathfrak a=\mathfrak a(\mathfrak g;\kappa,\lambda,0)$. A vector subspace $\mathfrak h\subseteq\mathfrak g$, regarded as an affine subspace through the origin, is a Lie subaffgebra if and only if
$$
[\mathfrak h,\mathfrak h]\subseteq\mathfrak h,
\quad
\kappa(\mathfrak h)\subseteq\mathfrak h,
\quad
\lambda(\mathfrak h)\subseteq\mathfrak h.
$$
\end{coro}

\begin{proof}
Apply Proposition~\ref{p1} with $a=0$ and $s=0$.
\end{proof}

\subsection{Rota-Baxter Lie algebras}\label{u7}

\begin{defn}
\textup{\cite{BGN2010,G2012}} Let $(\mathfrak g,[-,-])$ be a Lie algebra and $\theta\in\mathbf{k}$. A linear map $R:\mathfrak g\to\mathfrak g$ is a \textbf{Rota-Baxter operator of weight $\theta$} if, for all $x,y\in\mathfrak g$,
\begin{equation}\label{e6}
[R(x),R(y)]
=
R\big([R(x),y]+[x,R(y)]+\theta[x,y]\big).
\end{equation}

\end{defn}

For any $R\in\operatorname{End}(\mathfrak g)$ and $\theta\in\mathbf{k}$, set
\begin{equation}\label{e7}
[x,y]_{R,\theta}
=
[R(x),y]+[x,R(y)]+\theta[x,y].
\end{equation}
If $R$ satisfies Eq.~(\ref{e6}), then the operation in Eq.~(\ref{e7}) is a Lie bracket and $R:(\mathfrak g,[-,-]_{R,\theta})\to(\mathfrak g,[-,-])$ is a Lie algebra homomorphism~\cite{BGN2010}.

At weight $-1$, idempotent operators correspond to decompositions into Lie subalgebras.

\begin{prop}\label{p2}
\textup{\cite{G2012}} Let $\mathfrak g$ be a Lie algebra and let $R\in\operatorname{End}(\mathfrak g)$ satisfy $R^2=R$. Then $R$ is a Rota-Baxter operator of weight $-1$ if and only if $\operatorname{Im}R$ and $\ker R$ are Lie subalgebras of $\mathfrak g$.
\end{prop}

\begin{proof}
Since $R^2=R$, we have $\mathfrak g=\operatorname{Im}R\oplus\ker R$. For $x,y\in\operatorname{Im}R$, the Rota-Baxter identity gives $[x,y]=R([x,y])$. For $x,y\in\ker R$, it gives $R([x,y])=0$. Thus both summands are Lie subalgebras.

Conversely, suppose that both summands are Lie subalgebras. Write $x=x_++x_-$ and $y=y_++y_-$, with $x_\pm,y_\pm\in\mathfrak g_\pm$, where $\mathfrak g_+=\operatorname{Im}R$ and $\mathfrak g_-=\ker R$. Then
\begin{align*}
R\big([R(x),y]+[x,R(y)]-[x,y]\big)
&=R\big([x_+,y_+]-[x_-,y_-]\big)\\
&=[x_+,y_+]=[R(x),R(y)].
\end{align*}
This is the Rota-Baxter identity of weight $-1$.
\end{proof}

For an affine map $P:\mathfrak a\to\mathfrak a$ and $o\in\mathfrak a$, its linear part at $o$ is $R_o(x)=\langle o,P(o),P(x)\rangle$, viewed as a linear endomorphism of $T_o\mathfrak a$.

\begin{defn}\label{d3}
Let $\mathfrak a$ be a Lie affgebra and let $\theta\in\mathbf{k}$. An affine map $P:\mathfrak a\to\mathfrak a$ is called a \textbf{tangent Rota-Baxter operator of weight $\theta$} if $R_o$ is a Rota-Baxter operator of weight $\theta$ on the tangent Lie algebra $T_o\mathfrak a$ for some $o\in\mathfrak a$.
\end{defn}

\begin{prop}\label{p15}
The tangent Rota-Baxter condition holds at one origin if and only if it holds at every origin.
\end{prop}

\begin{proof}
For $o,q\in\mathfrak a$, the translation $\tau_o^q(x)=\langle x,o,q\rangle$ is a Lie algebra isomorphism $T_o\mathfrak a\to T_q\mathfrak a$. Since $P$ is affine, $R_q\tau_o^q=\tau_o^qR_o$. Thus $R_q$ and $R_o$ are conjugate by a Lie algebra isomorphism, which preserves the Rota-Baxter identity.
\end{proof}

In coordinates $\mathfrak a=\mathfrak a(\mathfrak g;\kappa,\lambda,s)$ and $P(x)=R(x)+p$, the map $P$ is tangent Rota-Baxter of weight $\theta$ if and only if $R$ is Rota-Baxter of weight $\theta$ on $\mathfrak g$. This condition imposes no restriction on $p$.

\subsection{Affine Rota-Baxter operators}\label{u9}

The coefficients in the ordinary weighted descendant sum to one precisely at weight $-1$. At this weight, its three bracket values therefore form an affine combination without any correction.

\begin{defn}\label{d1}
Let $(\mathfrak a,\{-,-\})$ be a Lie affgebra and let $P:\mathfrak a\to\mathfrak a$ be affine. For $x,y\in\mathfrak a$, define the \textbf{descendant operation} by
\begin{equation}\label{e8}
\{x,y\}_P
=
\left\langle
\{P(x),y\},
\{x,y\},
\{x,P(y)\}
\right\rangle.
\end{equation}
The map $P$ is called an \textbf{affine Rota-Baxter operator} if, for all $x,y\in\mathfrak a$,
\begin{equation}\label{e9}
P(\{x,y\}_P)=\{P(x),P(y)\}.
\end{equation}
\end{defn}

Eq.~(\ref{e8}) uses only $P$, the bracket and the heap operation, and is independent of an origin. Its coordinate form is as follows.

\begin{prop}\label{p3}
Let $\mathfrak a=\mathfrak a(\mathfrak g;\kappa,\lambda,s)$ be a Lie affgebra and let $P:\mathfrak a\to\mathfrak a$ be an affine map. Set $p=P(0)\in\mathfrak g$ and $R(x)=P(x)-P(0)$ for $x\in\mathfrak g$. Then $R\in\operatorname{End}(\mathfrak g)$ is the linear part of $P$ and $P(x)=R(x)+p$. The descendant operation is given by
\begin{equation}\label{e10}
\{x,y\}_P
=
[x,y]_{R,-1}
+\kappa R(x)
+(\lambda R+\operatorname{ad}_p)(y-x)
+s+\kappa(p).
\end{equation}
Its bilinear part is $[-,-]_{R,-1}$ as defined in Eq.~(\ref{e7}). This formula holds for every linear map $R$.
\end{prop}

\begin{proof}
Using Eq.~(\ref{e2}) and $P(x)=R(x)+p$, we have
\begin{align*}
\{P(x),y\}
&=[R(x),y]+[p,y]+\kappa R(x)+\kappa(p)+\lambda(y)-\lambda R(x)-\lambda(p)+s,\\
\{x,y\}
&=[x,y]+\kappa(x)+\lambda(y)-\lambda(x)+s,\\
\{x,P(y)\}
&=[x,R(y)]+[x,p]+\kappa(x)+\lambda R(y)+\lambda(p)-\lambda(x)+s.
\end{align*}
In the chosen coordinates, the heap operation in Eq.~(\ref{e8}) is $a-b+c$. The terms involving $\kappa$, $\lambda$ and $s$ therefore combine as follows:
\begin{align*}
&\kappa(R(x)+p)-\kappa(x)+\kappa(x)=\kappa R(x)+\kappa(p),\\
&\lambda(y-R(x)-p)-\lambda(y-x)+\lambda(R(y)+p-x)=\lambda R(y-x),\\
&s-s+s=s.
\end{align*}
Moreover, $[p,y]+[x,p]=[p,y-x]=\operatorname{ad}_p(y-x)$. Hence Eq.~(\ref{e7}) gives
\begin{align*}
\{x,y\}_P
&=\{P(x),y\}-\{x,y\}+\{x,P(y)\}\\
&=[R(x),y]+[x,R(y)]-[x,y]+[p,y]+[x,p]\\
&\quad+\kappa R(x)+\kappa(p)+\lambda R(y-x)+s\\
&=[x,y]_{R,-1}+\kappa R(x)+(\lambda R+\operatorname{ad}_p)(y-x)+s+\kappa(p),
\end{align*}
which is Eq.~(\ref{e10}).
\end{proof}

At weight $-1$, the affine Rota-Baxter identity also implies the affine Jacobi identity.

\begin{theorem}\label{t1}
Let $\mathfrak a$ be a Lie affgebra and let $P:\mathfrak a\to\mathfrak a$ be an affine map. If $P$ satisfies Eq.~(\ref{e9}), then $\{-,-\}_P$ is a Lie affgebra bracket and $P:(\mathfrak a,\{-,-\}_P)\to(\mathfrak a,\{-,-\})$ is a Lie affgebra homomorphism. Moreover, $P$ is a tangent Rota-Baxter operator of weight $-1$.
\end{theorem}

\begin{proof}
Choose an origin and write $\mathfrak a=\mathfrak a(\mathfrak g;\kappa,\lambda,s)$. Set $p=P(0)$ and $R(x)=P(x)-P(0)$, so that $R$ is linear and $P(x)=R(x)+p$. We first derive the consequences of Eq.~(\ref{e9}). By Proposition~\ref{p3}, we have
\begin{align*}
P(\{x,y\}_P)
&=R([x,y]_{R,-1})+R\kappa R(x)
+R(\lambda R+\operatorname{ad}_p)(y-x)\\
&\quad+R(s+\kappa(p))+p
\quad\text{(by Eq.~(\ref{e10}))},\\
\{P(x),P(y)\}
&=[R(x)+p,R(y)+p]+\kappa(R(x)+p)
+\lambda(R(y)-R(x))+s\\
&=[R(x),R(y)]+\kappa R(x)
+(\lambda+\operatorname{ad}_p)R(y-x)+s+\kappa(p).
\end{align*}
Eq.~(\ref{e9}) equates these two expressions. Evaluating at $x=y=0$ gives the constant identity below. After subtracting the constant terms, setting $y=x$ gives the identity involving $\kappa$, and setting $x=0$ gives the identity involving $\lambda$. Removing these linear terms leaves the bilinear identity. Thus
\begin{equation}\label{e11}
\begin{gathered}
R([x,y]_{R,-1})=[R(x),R(y)],\quad R\kappa R=\kappa R,\\
R(\lambda R+\operatorname{ad}_p)=(\lambda+\operatorname{ad}_p)R,\ \ 
R(s+\kappa(p))=s-p+\kappa(p).
\end{gathered}
\end{equation}
The first equality is the Rota-Baxter identity of weight $-1$. Consequently, $[-,-]_{R,-1}$ is a Lie bracket and $R:(\mathfrak g,[-,-]_{R,-1})\to(\mathfrak g,[-,-])$ is a Lie algebra homomorphism.

We next verify the structural identity for $\kappa R$ and $\lambda R+\operatorname{ad}_p$. By bilinearity, we have
\begin{align*}
[x-R(x),y-R(y)]
&=[x,y]-[x,R(y)]-[R(x),y]+[R(x),R(y)]\\
&=[R(x),R(y)]-\bigl([R(x),y]+[x,R(y)]-[x,y]\bigr)\\
&=[R(x),R(y)]-[x,y]_{R,-1}
\quad\text{(by Eq.~(\ref{e7}))}.
\end{align*}
Hence
\begin{equation}\label{e12}
[x,y]_{R,-1}=[R(x),R(y)]-[x-R(x),y-R(y)].
\end{equation}
It follows that
\begin{equation}\label{e13}
\begin{aligned}
(\operatorname{id}-R)([x,y]_{R,-1})
&=[x,y]_{R,-1}-[R(x),R(y)]
&&\text{(by Eq.~(\ref{e11}))}\\
&=-[x-R(x),y-R(y)]
&&\text{(by Eq.~(\ref{e12}))}.
\end{aligned}
\end{equation}
The linear identities in Eq.~(\ref{e11}) also give
\begin{equation}\label{e14}
\begin{aligned}
(\operatorname{id}-R)\kappa R
&=\kappa R-R\kappa R=0,\\
(\operatorname{id}-R)(\lambda R+\operatorname{ad}_p)
&=\lambda R+\operatorname{ad}_p-(\lambda+\operatorname{ad}_p)R
=\operatorname{ad}_p(\operatorname{id}-R).
\end{aligned}
\end{equation}
Using Eqs.~(\ref{e11})-(\ref{e14}) and $\operatorname{ad}_p(x)=[p,x]$, we obtain
\begin{align*}
&\quad (\lambda R+\operatorname{ad}_p)([x,y]_{R,-1})
-[(\lambda R+\operatorname{ad}_p)(x),y]_{R,-1}
+[x,\kappa R(y)]_{R,-1}
-[x,(\lambda R+\operatorname{ad}_p)(y)]_{R,-1}\\
&=\lambda([R(x),R(y)])-[\lambda R(x),R(y)]
+[R(x),\kappa R(y)]-[R(x),\lambda R(y)]\\
&\quad+[p,[R(x),R(y)]]-[[p,R(x)],R(y)]-[R(x),[p,R(y)]]\\
&\quad-[p,[x-R(x),y-R(y)]]
+[[p,x-R(x)],y-R(y)]\\
&\quad+[x-R(x),[p,y-R(y)]]\\
&=0\quad\text{(by Eq.~(\ref{e1}) and the Jacobi identity)}.
\end{align*}
In the last step, Eq.~(\ref{e1}) makes the first line on the right zero, while the next line and the final two lines vanish separately by the Jacobi identity. Therefore
\begin{align*}
(\lambda R+\operatorname{ad}_p)([x,y]_{R,-1})
=[(\lambda R+\operatorname{ad}_p)(x),y]_{R,-1}
-[x,\kappa R(y)]_{R,-1}
+[x,(\lambda R+\operatorname{ad}_p)(y)]_{R,-1}.
\end{align*}
This is the condition in Proposition~\ref{tstr} for the new bracket and endomorphisms.

We now verify the two affine identities directly. Eq.~(\ref{e10}) is biaffine and gives $\{x,x\}_P=\kappa R(x)+s+\kappa(p)$. Thus
\begin{align*}
\{x,y\}_P+\{y,x\}_P
&=[x,y]_{R,-1}+[y,x]_{R,-1}+\kappa R(x+y)
+(\lambda R+\operatorname{ad}_p)(y-x+x-y)+2s+2\kappa(p)\\
&=\kappa R(x)+\kappa R(y)+2s+2\kappa(p)\\
&=\{x,x\}_P+\{y,y\}_P,
\end{align*}
which is affine antisymmetry. Setting $y=x$ in the structural identity above yields
\begin{align*}
0&=[(\lambda R+\operatorname{ad}_p)(x),x]_{R,-1}
-[x,\kappa R(x)]_{R,-1}
+[x,(\lambda R+\operatorname{ad}_p)(x)]_{R,-1}\\
&=-[x,\kappa R(x)]_{R,-1}
\quad\text{(by skew symmetry)}.
\end{align*}

For the affine Jacobi identity, Eq.~(\ref{e10}) gives
\begin{align*}
\{y,z\}_P-\{x,x\}_P
&=[y,z]_{R,-1}+\kappa R(y-x)
+(\lambda R+\operatorname{ad}_p)(z-y).
\end{align*}
The constant $s+\kappa(p)$ cancels in this difference. Applying Eq.~(\ref{e10}) again, we have
\begin{align*}
\{x,\{y,z\}_P\}_P-\{x,\{x,x\}_P\}_P
&=[x,\{y,z\}_P-\{x,x\}_P]_{R,-1}
+(\lambda R+\operatorname{ad}_p)(\{y,z\}_P-\{x,x\}_P)\\
&=[x,[y,z]_{R,-1}]_{R,-1}
+[x,\kappa R(y-x)]_{R,-1}
+[x,(\lambda R+\operatorname{ad}_p)(z-y)]_{R,-1}\\
&\quad+(\lambda R+\operatorname{ad}_p)([y,z]_{R,-1})
+(\lambda R+\operatorname{ad}_p)\kappa R(y-x)
+(\lambda R+\operatorname{ad}_p)^2(z-y).
\end{align*}
Let $\sum_{\mathrm{cyc}}$ denote summation over $(x,y,z)$, $(y,z,x)$ and $(z,x,y)$. The last two terms cancel in this sum because
$$
(y-x)+(z-y)+(x-z)=0,\quad (z-y)+(x-z)+(y-x)=0.
$$
The nested bracket terms sum to zero by the Jacobi identity for $[-,-]_{R,-1}$. Hence
\begin{align*}
&\quad \sum_{\mathrm{cyc}}\bigl(\{x,\{y,z\}_P\}_P-\{x,\{x,x\}_P\}_P\bigr)\\
&=\sum_{\mathrm{cyc}}\Bigl(
[x,\kappa R(y-x)]_{R,-1}
+[x,(\lambda R+\operatorname{ad}_p)(z-y)]_{R,-1}
+(\lambda R+\operatorname{ad}_p)([y,z]_{R,-1})\Bigr)\\
&=\sum_{\mathrm{cyc}}\Bigl(
(\lambda R+\operatorname{ad}_p)([y,z]_{R,-1})
-[(\lambda R+\operatorname{ad}_p)(y),z]_{R,-1}\\
&\quad+[y,\kappa R(z)]_{R,-1}
-[y,(\lambda R+\operatorname{ad}_p)(z)]_{R,-1}\Bigr)
-\sum_{\mathrm{cyc}}[x,\kappa R(x)]_{R,-1}\\
&=0\quad\text{(by the structural identity above)}.
\end{align*}
The second equality uses skew symmetry and cyclic reindexing. The last equality uses the structural identity at $y,z$ and $[x,\kappa R(x)]_{R,-1}=0$. This proves the affine Jacobi identity and therefore shows that $\{-,-\}_P$ is a Lie affgebra bracket.

Finally, Eq.~(\ref{e9}) states exactly that $P$ preserves this bracket. The chosen origin was arbitrary, and the first equality in Eq.~(\ref{e11}) proves the assertion about the linear part of $P$ at every origin.
\end{proof}

The affine Rota-Baxter identity is therefore equivalent to the following coordinate conditions.

\begin{coro}\label{c2}
Let $\mathfrak a=\mathfrak a(\mathfrak g;\kappa,\lambda,s)$ be a Lie affgebra and let $P:\mathfrak a\to\mathfrak a$ be an affine map. Let $R\in\operatorname{End}(\mathfrak g)$ be its linear part and put $p=P(0)\in\mathfrak g$, so that $P(x)=R(x)+p$. Then $P$ is an affine Rota-Baxter operator if and only if the following identities hold for all $x,y\in\mathfrak g$:
\begin{align}
[R(x),R(y)]&=R([x,y]_{R,-1}),\quad R\kappa R=\kappa R,\label{e15}\\
R(\lambda R+\operatorname{ad}_p)&=(\lambda+\operatorname{ad}_p)R,\quad
R(s+\kappa(p))=s-p+\kappa(p).\label{e17}
\end{align}
\end{coro}

\begin{proof}
By Proposition~\ref{p3}, comparing the bilinear, linear and constant parts of Eq.~(\ref{e9}) gives Eqs.~(\ref{e15}) and~(\ref{e17}). Conversely, substituting these four equations into Eq.~(\ref{e10}) gives Eq.~(\ref{e9}). The descendant is a Lie affgebra by Theorem~\ref{t1}.
\end{proof}

\subsection{Idempotent fixed points}\label{u11}

The definition of an affine Rota-Baxter operator requires no preferred point. Factorization, however, uses a point fixed by $P$ and idempotent for the bracket. The obstruction to such a point is described below.

Choose an arbitrary origin and write $\mathfrak a=\mathfrak a(\mathfrak g;\kappa,\lambda,s)$ and $P(x)=R(x)+p$.
If the origin is moved from $0$ to a point $a\in\mathfrak g$, then the same Lie affgebra and affine map are described by new coefficients $\kappa_a,\lambda_a,s_a,R_a,p_a$.

\begin{lemma}\label{l2}
Under the canonical translation identification of the tangent spaces, we have
$$
\kappa_a=\kappa,\quad R_a=R, \quad \lambda_a=\lambda+\operatorname{ad}_a,\quad s_a=s-(\operatorname{id}-\kappa)a,\quad p_a=p-(\operatorname{id}-R)a.
$$
\end{lemma}

\begin{proof}
In the coordinates with new origin $a$, a point represented by $x\in\mathfrak g$ is the old point $a+x$. The new coordinate of the bracket value is therefore $\{a+x,a+y\}-a$. Expanding the standard form gives
\begin{align*}
\{a+x,a+y\}-a
&=[a+x,a+y]+\kappa(a+x)+\lambda(y-x)+s-a\\
&=[x,y]+[a,y-x]+\kappa(x)+\lambda(y-x)+\kappa(a)+s-a\\
&=[x,y]+\kappa(x)+(\lambda+\operatorname{ad}_a)(y-x)
+s-(\operatorname{id}-\kappa)a.
\end{align*}
This proves the formulas for $\kappa_a,\lambda_a$ and $s_a$. Likewise, $P(a+x)-a=R(x)+R(a)+p-a=R(x)+p-(\operatorname{id}-R)a$, which proves the formulas for $R_a$ and $p_a$.
\end{proof}

\begin{defn}\label{d4}
Let $P:\mathfrak a\to\mathfrak a$ be an affine map on a Lie affgebra $\mathfrak a$. A point $e\in\mathfrak a$ is an \textbf{idempotent fixed point of $P$} if $\{e,e\}=e$ and $P(e)=e$.
\end{defn}

Define the linear map
\begin{equation}\label{e36}
F_{\kappa,R}:\mathfrak g\longrightarrow\mathfrak g\oplus\mathfrak g,
\quad
F_{\kappa,R}(x)
=
\big((\operatorname{id}-\kappa)x,(\operatorname{id}-R)x\big).
\end{equation}

\begin{theorem}\label{t4}
The class
\begin{equation}\label{e37}
\Omega(\mathfrak a,P)
=
[(s,p)]
\in
\operatorname{coker}F_{\kappa,R}
=
(\mathfrak g\oplus\mathfrak g)/\operatorname{Im}F_{\kappa,R}
\end{equation}
is independent of the chosen origin, under the canonical translation identifications of the tangent spaces. Moreover, the following are equivalent:
\begin{enumerate}[label={\rm(\roman*)}]
\item $\Omega(\mathfrak a,P)=0$.
\item $P$ has an idempotent fixed point.
\item there exists $e\in\mathfrak g$ such that
\begin{equation}\label{e38}
(\operatorname{id}-\kappa)e=s,
\quad
(\operatorname{id}-R)e=p.
\end{equation}
\end{enumerate}
If these conditions hold and $e_0$ is one idempotent fixed point, then the set of all idempotent fixed points is
\begin{equation}\label{e39}
e_0+
\big(
\ker(\operatorname{id}-\kappa)
\cap
\ker(\operatorname{id}-R)
\big).
\end{equation}
In particular, it is an affine space with translation space $\ker(\operatorname{id}-\kappa)\cap\ker(\operatorname{id}-R)$.
\end{theorem}

\begin{proof}
Under a change of origin by $a$, Lemma~\ref{l2} identifies $\kappa_a$ with $\kappa$ and $R_a$ with $R$, so the linear map $F_{\kappa,R}$ is unchanged under the stated identifications. The constant terms satisfy
\begin{align*}
(s_a,p_a)
&=\bigl(s-(\operatorname{id}-\kappa)a,\ p-(\operatorname{id}-R)a\bigr)
\quad\text{(by Lemma~\ref{l2})}\\
&=(s,p)-F_{\kappa,R}(a)
\quad\text{(by Eq.~(\ref{e36}))}.
\end{align*}
Consequently, $[(s_a,p_a)]=[(s,p)]$ in $\operatorname{coker}F_{\kappa,R}$. This proves independence of the origin.

Suppose that (i) holds. By the definition of the quotient in Eq.~(\ref{e37}), the equality $[(s,p)]=0$ means that $(s,p)$ belongs to $\operatorname{Im}F_{\kappa,R}$. Thus there is $e\in\mathfrak g$ with $F_{\kappa,R}(e)=(s,p)$, which is exactly (iii). These equations give
$$
\{e,e\}=[e,e]+\kappa(e)+\lambda(e-e)+s =\kappa(e)+s=e, \quad P(e)=R(e)+p=e.
$$
Hence $e$ is an idempotent fixed point of $P$ and (ii) holds.

Conversely, if (ii) holds, let $e$ be an idempotent fixed point. The equations $\{e,e\}=e$ and $P(e)=e$ give $(\operatorname{id}-\kappa)e=s$ and $(\operatorname{id}-R)e=p$, respectively. Thus $F_{\kappa,R}(e)=(s,p)$, proving (iii), and $(s,p)\in\operatorname{Im}F_{\kappa,R}$, proving (i).

Finally, if $e_0$ is one solution, linearity gives
$$
F_{\kappa,R}(e)=(s,p) \Longleftrightarrow F_{\kappa,R}(e-e_0)=0 \Longleftrightarrow e-e_0\in \ker(\operatorname{id}-\kappa)\cap\ker(\operatorname{id}-R).
$$
This proves Eq.~(\ref{e39}) and identifies the translation space of the set of idempotent fixed points.
\end{proof}

\begin{remark}\label{r1}
Theorem~\ref{t4} applies to every affine map $P$. Neither the Rota-Baxter identity nor idempotency is required. Once one idempotent fixed point exists, all others are obtained by adding vectors fixed by both $\kappa$ and $R$, as in Eq.~(\ref{e39}).
\end{remark}

The obstruction is also preserved by isomorphisms that intertwine the affine maps.

\begin{prop}\label{p6}
Let $\mathfrak a$ and $\mathfrak a'$ be Lie affgebras with affine maps $P:\mathfrak a\to\mathfrak a$ and $P':\mathfrak a'\to\mathfrak a'$, respectively, and let $\varphi:\mathfrak a\to\mathfrak a'$ be a Lie affgebra isomorphism satisfying $\varphi P=P'\varphi$. Choose origins and use the notation of Proposition~\ref{thom}. Write $P(x)=R(x)+p$, $P'(x)=R'(x)+p'$ and $\varphi(x)=\psi(x)+q$. Then $\psi$ induces an isomorphism
$$
\overline{\psi\oplus\psi}:
\operatorname{coker}F_{\kappa,R}
\longrightarrow
\operatorname{coker}F_{\kappa',R'}
$$
which sends $\Omega(\mathfrak a,P)$ to $\Omega(\mathfrak a',P')$.
\end{prop}

\begin{proof}
The homomorphism conditions in Eq.~(\ref{e3}) and the relation $\varphi P=P'\varphi$ give
$$
\psi\kappa=\kappa'\psi,\quad \psi R=R'\psi, \quad \psi s=s'-(\operatorname{id}-\kappa')q,\quad \psi p=p'-(\operatorname{id}-R')q.
$$
It follows that $(\psi\oplus\psi)F_{\kappa,R}(x)=F_{\kappa',R'}(\psi x)$, so $\psi\oplus\psi$ descends to the stated map on cokernels. Moreover, $(\psi\oplus\psi)(s,p)=(s',p')-F_{\kappa',R'}(q)$, which proves that the obstruction classes correspond.
\end{proof}

Taking only the first component gives an obstruction intrinsic to the Lie affgebra itself.

\begin{coro}\label{c6}
The class $\omega(\mathfrak a)=[s]\in\operatorname{coker}(\operatorname{id}-\kappa)$ is independent of the chosen origin. It vanishes if and only if $\mathfrak a$ has an idempotent point. If nonempty, the set of idempotent points is an affine space with translation space $\ker(\operatorname{id}-\kappa)$.
\end{coro}

\begin{proof}
By Lemma~\ref{l2}, changing the origin replaces $s$ by $s-(\operatorname{id}-\kappa)a$, which leaves its cokernel class unchanged. The remaining assertions follow by solving $(\operatorname{id}-\kappa)e=s$.
\end{proof}

\begin{exam}\label{x2}
Let $\mathfrak g$ be a Lie algebra, take $s\in\mathfrak g\setminus\{0\}$, and set $\kappa=\lambda=\operatorname{id}_{\mathfrak g}$. Then $\{x,y\}=[x,y]+y+s$ defines a Lie affgebra. Since $\operatorname{id}-\kappa=0$, the class $\omega(\mathfrak a)$ is represented by the nonzero element $s$ and therefore does not vanish. Equivalently, $\{e,e\}=e+s\neq e$ for every point $e$. Even the identity affine Rota-Baxter operator on this Lie affgebra therefore admits no idempotent fixed point. This example shows that the existence of an idempotent fixed point must be checked separately.
\end{exam}

\begin{remark}\label{r2}
When $\Omega(\mathfrak a,P)=0$, choosing an idempotent fixed point $e$ as origin removes both constant terms: the Lie affgebra becomes $\mathfrak a(T_e\mathfrak a;\kappa_e,\lambda_e,0)$ and the affine map $P$ becomes linear. Conversely, any origin at which both constant terms vanish is an idempotent fixed point.
\end{remark}

\subsection{Weighted affine descendants}\label{u10}

Fix $\theta\in\mathbf{k}$ throughout this subsection. The coefficients of the three bracket terms in the ordinary weighted descendant sum to $2+\theta$. Subtracting each diagonal bracket value with coefficient $(\theta+1)/2$ makes this sum equal to one.

\begin{defn}\label{d2}
Let $(\mathfrak a,\{-,-\})$ be a Lie affgebra and let $P:\mathfrak a\to\mathfrak a$ be an affine map. For $x,y\in\mathfrak a$, define the \textbf{weighted affine descendant} by
\begin{equation}\label{e19}
\{x,y\}_{P,\theta} =\{P(x),y\}+\{x,P(y)\}+\theta\{x,y\} -\frac{\theta+1}{2}\{x,x\}-\frac{\theta+1}{2}\{y,y\}.
\end{equation}
The expression on the right is interpreted as the affine combination of the five displayed points.
\end{defn}

Indeed, $1+1+\theta-(\theta+1)/2-(\theta+1)/2=1$, so the operation in Eq.~(\ref{e19}) is independent of the origin. For a symmetric correction $-c\{x,x\}-c\{y,y\}$, the requirement that the coefficients sum to one is $2+\theta-2c=1$, which gives $c=(\theta+1)/2$. At $\theta=-1$, both correction terms vanish and
\begin{equation}\label{e20}
\{x,y\}_{P,-1}=\{x,y\}_P.
\end{equation}

\begin{lemma}\label{l1}
For every Lie affgebra $\mathfrak a$, the map $x\mapsto\{x,x\}$ from $\mathfrak a$ to itself is affine.
\end{lemma}

\begin{proof}
Choose an origin and write $\mathfrak a=\mathfrak a(\mathfrak g;\kappa,\lambda,s)$. Then $\{x,x\}=\kappa(x)+s$, which is affine.
\end{proof}

\begin{prop}\label{p4}
For every affine map $P:\mathfrak a\to\mathfrak a$ on a Lie affgebra $\mathfrak a$, the operation $\{-,-\}_{P,\theta}$ is biaffine and satisfies affine antisymmetry.
\end{prop}

\begin{proof}
Biaffineness follows from Lemma~\ref{l1}. Choose an arbitrary origin. Affine antisymmetry of the original bracket means $\{a,b\}+\{b,a\}=\{a,a\}+\{b,b\}$. Applying this identity twice, and then using $\{x,x\}_{P,\theta}=\{P(x),P(x)\}$, gives
$$
\{x,y\}_{P,\theta}+\{y,x\}_{P,\theta}
=
\{P(x),P(x)\}+\{P(y),P(y)\}
=
\{x,x\}_{P,\theta}+\{y,y\}_{P,\theta}.
$$
\end{proof}

We compute the coefficients of the weighted descendant.

\begin{prop}\label{t2}
Let $\mathfrak a=\mathfrak a(\mathfrak g;\kappa,\lambda,s)$ be a Lie affgebra and let $P:\mathfrak a\to\mathfrak a$ be an affine map. Let $R\in\operatorname{End}(\mathfrak g)$ be its linear part and put $p=P(0)\in\mathfrak g$, so that $P(x)=R(x)+p$. Then
\begin{equation}\label{e21}
\{x,y\}_{P,\theta}
=
[x,y]_{R,\theta}
+\kappa_P(x)
+\lambda_P(y-x)
+s_P,
\end{equation}
where
\begin{equation}\label{e22}
\kappa_P=\kappa R,\quad s_P=s+\kappa(p), \quad \lambda_P=\lambda R+(\theta+1)\lambda+\operatorname{ad}_p -\frac{\theta+1}{2}\kappa.
\end{equation}
The notation $\lambda_P$ suppresses its dependence on the fixed weight $\theta$.

Its bilinear part is $[-,-]_{R,\theta}$ as defined in Eq.~(\ref{e7}), which is a Lie bracket whenever $P$ is a tangent Rota-Baxter operator of weight $\theta$.
\end{prop}

\begin{proof}
We first isolate the diagonal correction in Eq.~(\ref{e19}). By Eq.~(\ref{e2}), we have
\begin{align*}
\{x,y\}-\frac12\{x,x\}-\frac12\{y,y\}
&=[x,y]+\kappa(x)+\lambda(y-x)+s
-\frac12(\kappa(x)+s)-\frac12(\kappa(y)+s)\\
&=[x,y]+\left(\lambda-\frac12\kappa\right)(y-x).
\end{align*}
Then it follows from Eqs.~(\ref{e8}) and~(\ref{e19}) that
\begin{align*}
\{x,y\}_{P,\theta}
&=\{x,y\}_P+(\theta+1)
\left(\{x,y\}-\frac12\{x,x\}-\frac12\{y,y\}\right)\\
&=[x,y]_{R,-1}+(\theta+1)[x,y]+\kappa R(x)+s+\kappa(p)\\
&\quad+(\lambda R+\operatorname{ad}_p)(y-x)
+(\theta+1)\left(\lambda-\frac12\kappa\right)(y-x)
\quad\text{(by Eq.~(\ref{e10}))}\\
&=[x,y]_{R,\theta}+\kappa R(x)+s+\kappa(p)\\
&\quad+\left(\lambda R+(\theta+1)\lambda+\operatorname{ad}_p
-\frac{\theta+1}{2}\kappa\right)(y-x)
\quad\text{(by Eq.~(\ref{e7}))}\\
&=[x,y]_{R,\theta}+\kappa_P(x)+\lambda_P(y-x)+s_P
\quad\text{(by Eq.~(\ref{e22}))}.
\end{align*}
This proves Eq.~(\ref{e21}) and identifies all its coefficients. The assertion about the bilinear part is the ordinary Rota-Baxter construction recalled after Eq.~(\ref{e7}).
\end{proof}

We now determine when the weighted descendant is a Lie affgebra bracket and $P$ is a homomorphism to the original Lie affgebra. The latter requirement is the identity
\begin{equation}\label{e25}
P(\{x,y\}_{P,\theta})=\{P(x),P(y)\},\quad x,y\in\mathfrak a.
\end{equation}

For general $\theta$, Eq.~(\ref{e25}) alone need not imply the affine Jacobi identity. For $P=R+p$ on $\mathfrak a(\mathfrak g;\kappa,\lambda,s)$, define the bilinear map $\mathcal J_P:\mathfrak g\times\mathfrak g\to\mathfrak g$ by
\begin{equation}\label{e26}
\mathcal J_P(x,y) =\lambda_P([x,y]_{R,\theta})-[\lambda_P(x),y]_{R,\theta} +[x,\kappa_P(y)]_{R,\theta}-[x,\lambda_P(y)]_{R,\theta}.
\end{equation}
Here $\kappa_P$ and $\lambda_P$ are given by Eq.~(\ref{e22}), and $\mathcal J_P=0$ means that $\mathcal J_P(x,y)=0$ for all $x,y\in\mathfrak g$. The fixed weight $\theta$ is also suppressed in $\mathcal J_P$. If $R$ is a Rota-Baxter operator of weight $\theta$, then $[-,-]_{R,\theta}$ is a Lie bracket and $\mathcal J_P=0$ is precisely the structural identity for $(\kappa_P,\lambda_P)$ with respect to this bracket.

\begin{theorem}\label{t3}
Let $\mathfrak a=\mathfrak a(\mathfrak g;\kappa,\lambda,s)$ be a Lie affgebra and let $P:\mathfrak a\to\mathfrak a$ be an affine map. Let $R\in\operatorname{End}(\mathfrak g)$ be its linear part and put $p=P(0)\in\mathfrak g$, so that $P(x)=R(x)+p$. Then $\{-,-\}_{P,\theta}$ is a Lie affgebra bracket and $P:(\mathfrak a,\{-,-\}_{P,\theta})\to(\mathfrak a,\{-,-\})$ is a Lie affgebra homomorphism if and only if
\begin{enumerate}[label={\rm(\roman*)}]
\item $P$ satisfies the affine identity in Eq.~(\ref{e25}).
\item $\mathcal J_P=0$.
\end{enumerate}
\end{theorem}

\begin{proof}
Suppose first that the descendant is a Lie affgebra and $P$ is a homomorphism. Then Eq.~(\ref{e25}) holds. By Eq.~(\ref{e21}), the bilinear part is $[-,-]_{R,\theta}$ and the remaining coefficients are $\kappa_P$, $\lambda_P$ and $s_P$. Proposition~\ref{tstr} therefore gives
$$
\lambda_P([x,y]_{R,\theta})
=[\lambda_P(x),y]_{R,\theta}
-[x,\kappa_P(y)]_{R,\theta}
+[x,\lambda_P(y)]_{R,\theta},
$$
which is $\mathcal J_P(x,y)=0$ by Eq.~(\ref{e26}).

Conversely, assume (i) and (ii). Expanding the two sides of Eq.~(\ref{e25}), we have
\begin{align*}
P(\{x,y\}_{P,\theta})
&=R([x,y]_{R,\theta})+R\kappa_P(x)+R\lambda_P(y-x)+R(s_P)+p
\quad\text{(by Eq.~(\ref{e21}))},\\
\{P(x),P(y)\}
&=[R(x)+p,R(y)+p]+\kappa(R(x)+p)+\lambda R(y-x)+s\\
&=[R(x),R(y)]+\kappa R(x)
+(\lambda+\operatorname{ad}_p)R(y-x)+s+\kappa(p).
\end{align*}
The constant terms agree by setting $x=y=0$. Subtracting them and setting $y=x$ identifies the $\kappa$ terms; setting $x=0$ identifies the $\lambda$ terms. The remaining equality is the bilinear identity. Thus
\begin{equation}\label{e27}
R([x,y]_{R,\theta})=[R(x),R(y)], \quad R\kappa_P=\kappa R, \quad R\lambda_P=(\operatorname{ad}_p+\lambda)R, \quad R(s_P)=s-p+\kappa(p).
\end{equation}
The first identity makes $R$ an ordinary Rota-Baxter operator of weight $\theta$, so $[-,-]_{R,\theta}$ is a Lie bracket. By Eq.~(\ref{e26}), condition (ii) is exactly the structural identity displayed above. Proposition~\ref{tstr}, applied to this Lie bracket, $\kappa_P,\lambda_P$ and $s_P$, now proves that Eq.~(\ref{e21}) is a Lie affgebra bracket. Condition (i) makes $P$ a homomorphism from this descendant to the original Lie affgebra.
\end{proof}

\begin{prop}\label{p5}
With the notation of Theorem~\ref{t3}, assume that $P$ satisfies Eq.~(\ref{e25}). Then
\begin{equation}\label{e28}
R\mathcal J_P(x,y)=0
\end{equation}
for all $x,y\in\mathfrak g$. Equivalently, $\operatorname{Im}\mathcal J_P\subseteq\ker R$.
\end{prop}

\begin{proof}
By Eq.~(\ref{e27}), $R$ is a Lie algebra homomorphism from $(\mathfrak g,[-,-]_{R,\theta})$ to $(\mathfrak g,[-,-])$, with $R\kappa_P=\kappa R$ and $R\lambda_P=(\lambda+\operatorname{ad}_p)R$. Applying $R$ to Eq.~(\ref{e26}), we obtain
\begin{align*}
R\mathcal J_P(x,y)
&=(\lambda+\operatorname{ad}_p)[R(x),R(y)]-[(\lambda+\operatorname{ad}_p)R(x),R(y)]\\
&\quad+[R(x),\kappa R(y)]-[R(x),(\lambda+\operatorname{ad}_p)R(y)]\\
&=[p,[R(x),R(y)]]-[[p,R(x)],R(y)]-[R(x),[p,R(y)]]
\quad\text{(by Eq.~(\ref{e1}))}\\
&=0\quad\text{(by the Jacobi identity)}.
\end{align*}
\end{proof}

\begin{coro}\label{c3}
With the notation of Theorem~\ref{t3}, suppose that $R$ is injective and $P$ satisfies Eq.~(\ref{e25}). Then $\{-,-\}_{P,\theta}$ is a Lie affgebra bracket and $P$ is a Lie affgebra homomorphism from the descendant to the original Lie affgebra.
\end{coro}

\begin{proof}
Proposition~\ref{p5} gives $R\mathcal J_P=0$, and injectivity implies $\mathcal J_P=0$. Theorem~\ref{t3} applies.
\end{proof}

\begin{coro}\label{c4}
Let $\theta\in\mathbf{k}$, let $\mathfrak a=\mathfrak a(\mathfrak g;\kappa,\lambda,s)$ be a Lie affgebra, and let $P:\mathfrak a\to\mathfrak a$ be an affine map. Let $R\in\operatorname{End}(\mathfrak g)$ be the linear part of $P$, put $p=P(0)$, and let $\kappa_P,\lambda_P$ be given by Eq.~(\ref{e22}). Thus $P(x)=R(x)+p$. Then $\{-,-\}_{P,\theta}$ is a Lie affgebra bracket and $P$ is a homomorphism to the original Lie affgebra if and only if the following identities hold for all $x,y\in\mathfrak g$:
\begin{align}
[R(x),R(y)]&=R([x,y]_{R,\theta}),\label{e29}\\
\lambda_P([x,y]_{R,\theta})
&=[\lambda_P(x),y]_{R,\theta}-[x,\kappa_P(y)]_{R,\theta}+[x,\lambda_P(y)]_{R,\theta},\label{e30}\\
R\kappa R&=\kappa R,\quad
R\lambda_P=(\operatorname{ad}_p+\lambda)R,\quad
R(s+\kappa(p))=s-p+\kappa(p).\label{e31}
\end{align}
\end{coro}

\begin{proof}
By Theorem~\ref{t3}, the descendant is a Lie affgebra and $P$ is a homomorphism if and only if Eq.~(\ref{e25}) holds for all $x,y\in\mathfrak g$ and $\mathcal J_P=0$. The coefficient comparison in Eq.~(\ref{e27}), together with Eq.~(\ref{e22}), shows that Eq.~(\ref{e25}) is equivalent to Eqs.~(\ref{e29}) and~(\ref{e31}). Eq.~(\ref{e26}) identifies $\mathcal J_P=0$ with Eq.~(\ref{e30}).
\end{proof}

At weight $-1$, Eq.~(\ref{e20}) identifies the weighted descendant with the operation in Definition~\ref{d1}. Theorem~\ref{t1} then shows that Eq.~(\ref{e25}) alone implies the affine Jacobi identity. The following example shows that this implication fails at weight $1$.

\begin{exam}\label{x1}
Let $\mathfrak g$ have basis $e_1,e_2$ with $[e_1,e_2]=e_2$. Take $\kappa=\operatorname{id}$, $s=0$, and define $\lambda(e_1)=e_1$ and $\lambda(e_2)=0$. Then Eq.~(\ref{e1}) holds, so this choice defines a Lie affgebra. Let $\theta=1$ and let $P$ be the constant map with value $0$. Both sides of Eq.~(\ref{e25}) equal zero. However, $\kappa_P=0$ and $\lambda_P=2\lambda-\operatorname{id}$. Hence $\lambda_P(e_1)=e_1$, $\lambda_P(e_2)=-e_2$ and, since $[-,-]_{R,1}=[-,-]$, we obtain $\mathcal J_P(e_1,e_2)=-e_2\neq0$.
Thus the affine identity alone does not imply that the weighted descendant is a Lie affgebra bracket at weight $1$.
\end{exam}

\section{Affine factorization}\label{s3}

An idempotent affine map decomposes an affine space into its image and a fibre over a fixed point. For a Lie affgebra, both factors must also be closed under the bracket. We identify the conditions for this closure and prove the correspondence with ordered affine Lie factorizations. The last subsection constructs affine doubles at nonzero weight and identifies their tangent Lie algebras.

\subsection{Identity and constant operators}\label{u12}

Throughout Sections~\ref{u12}-\ref{u14}, affine Rota-Baxter operators have weight $-1$ in the sense of Definition~\ref{d1}. Section~\ref{u15} treats affine doubles at arbitrary nonzero weight.

\begin{coro}
The identity map of every Lie affgebra is an affine Rota-Baxter operator.
\end{coro}

\begin{proof}
For $P=\operatorname{id}$, Eq.~(\ref{e8}) gives $\{x,y\}_P=\{x,y\}$, and Eq.~(\ref{e9}) follows directly.
\end{proof}

Recall that a point $e\in\mathfrak a$ is idempotent when $\{e,e\}=e$. By Corollary~\ref{c6}, the existence of such a point is equivalent to the vanishing of $\omega(\mathfrak a)$.

\begin{prop}\label{p7}
Let $P_e$ be the constant map with value $e$. Then $P_e$ is an affine Rota-Baxter operator if and only if $e$ is an idempotent point.
\end{prop}

\begin{proof}
The two sides of Eq.~(\ref{e9}) are $e$ and $\{e,e\}$, respectively. Hence the affine Rota-Baxter identity holds exactly when $\{e,e\}=e$. Theorem~\ref{t1} shows that the descendant is a Lie affgebra.
\end{proof}

\subsection{Images and affine kernels}\label{u13}

\begin{prop}\label{p8}
Let $P:\mathfrak a\to\mathfrak a$ be an affine map on a Lie affgebra $(\mathfrak a,\{-,-\})$ satisfying Eq.~(\ref{e25}) at weight $\theta$. Then $\operatorname{Im}P$ is a Lie subaffgebra of the original Lie affgebra. In particular, this holds for every affine Rota-Baxter operator in Definition~\ref{d1}.
\end{prop}

\begin{proof}
The image of an affine map is an affine subspace. Let $a=P(x)$ and $b=P(y)$ be two points of $\operatorname{Im}P$. By the defining homomorphism identity, $\{a,b\}=\{P(x),P(y)\}=P(\{x,y\}_{P,\theta})$. Hence $\{a,b\}\in\operatorname{Im}P$.
\end{proof}

The fibre over a fixed point requires a separate closure condition.

For the following definition and lemma, let $P:\mathfrak a\to\mathfrak a$ be an affine map satisfying $P^2=P$. Then $\operatorname{Im}P=\operatorname{Fix}P$, where $\operatorname{Fix}P:=\{x\in\mathfrak a\mid P(x)=x\}$.
Let $e\in\operatorname{Im}P$. Thus $P(e)=e$.

\begin{defn}
The \textbf{affine kernel of $P$ at $e$} is $\operatorname{Ker}_eP=P^{-1}(e)$.
The \textbf{complementary affine map at $e$} is $Q_e(x)=\langle x,P(x),e\rangle$.
\end{defn}

The two affine maps satisfy the following projection identities.

\begin{lemma}\label{l3}
Let $P$ be affine and idempotent, and let $e\in\operatorname{Im}P$. Then
\begin{enumerate}[label={\rm(\roman*)}]
\item $Q_e$ is affine and idempotent.
\item $\operatorname{Im}Q_e=\operatorname{Ker}_eP$.
\item $P Q_e=Q_eP=e$, where $e$ on the right denotes the constant map with value $e$.
\item for every $x\in\mathfrak a$,
\begin{equation}\label{e42}
x=\langle P(x),e,Q_e(x)\rangle.
\end{equation}
\item the representation in Eq.~(\ref{e42}) is unique with $P(x)\in\operatorname{Im}P$ and $Q_e(x)\in\operatorname{Ker}_eP$.
\end{enumerate}
\end{lemma}

\begin{proof}
The map $Q_e$ is affine because it is the affine combination $x-P(x)+e$. Since $P$ preserves heap operations,
$$
P(Q_e(x)) =P(\langle x,P(x),e\rangle) =\langle P(x),P^2(x),P(e)\rangle =\langle P(x),P(x),e\rangle=e.
$$
Hence $\operatorname{Im}Q_e\subseteq\operatorname{Ker}_eP$. If $y\in\operatorname{Ker}_eP$, then $Q_e(y)=\langle y,e,e\rangle=y$, so equality holds and $Q_e^2=Q_e$. Moreover, $Q_e(P(x))=\langle P(x),P^2(x),e\rangle=e$.
This proves (i)-(iii).

The heap identities give $\langle P(x),e,Q_e(x)\rangle=\langle P(x),e,\langle x,P(x),e\rangle\rangle=x$, which proves (iv).

For uniqueness, suppose $x=\langle a,e,b\rangle$, where $a\in\operatorname{Im}P$ and $b\in\operatorname{Ker}_eP$. Then $P(x)=\langle P(a),P(e),P(b)\rangle=\langle a,e,e\rangle=a$, and hence $Q_e(x)=\langle x,P(x),e\rangle=b$.
\end{proof}

Lemma~\ref{l3} gives a decomposition into affine subspaces. Their closure under the Lie affgebra bracket requires further conditions.

\begin{prop}\label{p9}
Let $P:\mathfrak a\to\mathfrak a$ be an affine map such that $\{-,-\}_{P,\theta}$ is a Lie affgebra bracket and $P:(\mathfrak a,\{-,-\}_{P,\theta})\to(\mathfrak a,\{-,-\})$ is a homomorphism. If $e\in\operatorname{Im}P$ satisfies $\{e,e\}=e$, then $P^{-1}(e)$ is a Lie subaffgebra of the descendant Lie affgebra. In particular, for an idempotent affine Rota-Baxter operator the affine kernel at an idempotent fixed point is a Lie subaffgebra of $(\mathfrak a,\{-,-\}_P)$.
\end{prop}

\begin{proof}
The singleton $\{e\}$ is a Lie subaffgebra of the original Lie affgebra because $\{e,e\}=e$. Its nonempty inverse image under the homomorphism $P$ is an affine subspace closed under the descendant bracket.
\end{proof}

The corresponding statement for the original bracket is false.

\begin{exam}\label{x3}
Let $\mathfrak g$ be the abelian Lie algebra of dimension 2 with basis $u,v$. Define $\kappa(u)=0$, $\kappa(v)=u$, $\lambda=0$ and $s=0$.
Since $\mathfrak g$ is abelian, Eq.~(\ref{e1}) holds automatically, and $\{x,y\}=\kappa(x)$ defines a Lie affgebra.

Let $R(u)=u$ and $R(v)=0$.
Then $R^2=R$. Since the tangent Lie algebra is abelian, $R$ is a Rota-Baxter operator of weight $-1$. Moreover, $\kappa R=0$ gives the second identity in Eq.~(\ref{e15}), while $\lambda=0$ and $s=p=0$ give Eq.~(\ref{e17}). Hence $R$, regarded as an affine map, is an idempotent affine Rota-Baxter operator.

Its image is $\mathbf{k}u$, which is a Lie subaffgebra. However, $\operatorname{Ker}_0R=\mathbf{k}v$ is not a Lie subaffgebra of the original Lie affgebra, because $\{v,v\}=\kappa(v)=u\notin\mathbf{k}v$.
Thus the affine Rota-Baxter identity and idempotency alone do not imply affine factorization into two Lie subaffgebras.
\end{exam}

\subsection{Splitting and affine factorization}\label{u14}

We now characterize closure of the affine kernel using coordinates at an idempotent fixed point.

Let $P$ be an idempotent affine Rota-Baxter operator, and let $e\in\operatorname{Im}P$ be an idempotent point. Choose $e$ as the origin. Since $P(e)=e$, the map $P$ becomes a linear projection $R:T_e\mathfrak a\to T_e\mathfrak a$.
Since $e$ is idempotent, the constant term in the standard description of the Lie affgebra at $e$ vanishes. Thus $\mathfrak a=\mathfrak a(\mathfrak g;\kappa_e,\lambda_e,0)$ with $\mathfrak g=T_e\mathfrak a$.

Set $\mathfrak g_+=\operatorname{Im}R$ and $\mathfrak g_-=\ker R$.
Since $R$ is an idempotent Rota-Baxter operator of weight $-1$, Proposition~\ref{p2} gives $\mathfrak g=\mathfrak g_+\oplus\mathfrak g_-$ with both summands Lie subalgebras.

\begin{theorem}\label{t5}
With the notation above, the following statements are equivalent:
\begin{enumerate}[label={\rm(\roman*)}]
\item $\operatorname{Ker}_eP$ is a Lie subaffgebra of the original Lie affgebra $\mathfrak a$.
\item
$$
\kappa_e(\mathfrak g_-)
\subseteq
\mathfrak g_-,
\quad
\lambda_e(\mathfrak g_-)
\subseteq
\mathfrak g_-.
$$
\item
$$
R\kappa_e=\kappa_eR,
\quad
R\lambda_e=\lambda_eR.
$$
\end{enumerate}
Whenever these conditions hold, $\operatorname{Im}P$ and $\operatorname{Ker}_eP$ are both Lie subaffgebras of the original Lie affgebra.
\end{theorem}

\begin{proof}
With $e$ as origin, we have $\operatorname{Im}P=\mathfrak g_+$ and $\operatorname{Ker}_eP=\mathfrak g_-$. Both are Lie subalgebras of $(\mathfrak g,[-,-])$ by Proposition~\ref{p2}.

Suppose (i) holds. For $x\in\mathfrak g_-$, closure under the affine bracket gives
$$
\kappa_e(x)=\{x,x\}\in\mathfrak g_-,
\quad \lambda_e(x)=\{0,x\}\in\mathfrak g_-
\quad\text{(by Eq.~(\ref{e2}) with $s=0$)}.
$$
This proves (ii). Conversely, if (ii) holds and $x,y\in\mathfrak g_-$, then
$$
\{x,y\}=[x,y]+\kappa_e(x)+\lambda_e(y-x)\in\mathfrak g_-,
$$
because all three terms lie in $\mathfrak g_-$. Thus (i) and (ii) are equivalent.

By Proposition~\ref{p8}, $\mathfrak g_+$ is a Lie subaffgebra. Applying the same diagonal and zero-argument evaluations to $\mathfrak g_+$ shows that $\kappa_e$ and $\lambda_e$ preserve it. Therefore $R\kappa_eR=\kappa_eR$ and $R\lambda_eR=\lambda_eR$. Under (ii), the maps also preserve $\ker R$, so
$$
R\kappa_e(\operatorname{id}-R)=0,\quad R\lambda_e(\operatorname{id}-R)=0.
$$
Consequently,
\begin{align*}
R\kappa_e
&=R\kappa_eR+R\kappa_e(\operatorname{id}-R)=\kappa_eR,\\
R\lambda_e
&=R\lambda_eR+R\lambda_e(\operatorname{id}-R)=\lambda_eR.
\end{align*}
This proves (iii). If (iii) holds and $x\in\ker R$, then
$$
R\kappa_e(x)=\kappa_eR(x)=0,\quad
R\lambda_e(x)=\lambda_eR(x)=0,
$$
so both endomorphisms preserve $\mathfrak g_-$, proving (ii). The final assertion follows from Proposition~\ref{p8} and (i).
\end{proof}

\begin{defn}\label{d5}
Let $P:\mathfrak a\to\mathfrak a$ be an idempotent affine Rota-Baxter operator and let $e$ be an idempotent fixed point of $P$. We say that $P$ is \textbf{split at $e$} if the fibre $P^{-1}(e)$ is a Lie subaffgebra of the original Lie affgebra $\mathfrak a$.
\end{defn}

By Theorem~\ref{t5}, this is equivalent to the two commutation relations stated there. Both affine factors provided by Lemma~\ref{l3} are then Lie subaffgebras of the original Lie affgebra.

\begin{defn}\label{d6}
Let $\mathfrak a_+$ and $\mathfrak a_-$ be affine subspaces of an affine space $\mathfrak a$ with $\mathfrak a_+\cap\mathfrak a_- = \{e\}$.
We write $\mathfrak a=\mathfrak a_+\oplus_e\mathfrak a_-$ if the map
$$
\mathfrak a_+\times\mathfrak a_-
\longrightarrow
\mathfrak a,
\quad
(a,b)\longmapsto\langle a,e,b\rangle
$$
is bijective.
If $\mathfrak a$ is a Lie affgebra and both $\mathfrak a_+$ and $\mathfrak a_-$ are Lie subaffgebras, we call this an \textbf{ordered affine Lie factorization}.
\end{defn}

If $\mathfrak a_+$ and $\mathfrak a_-$ are Lie subaffgebras and their intersection is $\{e\}$, then necessarily $\{e,e\}\in\mathfrak a_+\cap\mathfrak a_-=\{e\}$.
Hence the intersection point of an affine Lie factorization is automatically idempotent.

We can now establish the correspondence with ordered affine Lie factorizations.

\begin{theorem}\label{t6}
Let $(\mathfrak a,\{-,-\})$ be a Lie affgebra. There is a bijection between the following sets:
\begin{enumerate}[label={\rm(\alph*)}]
\item ordered affine Lie factorizations $\mathfrak a=\mathfrak a_+\oplus_e\mathfrak a_-$.
\item pairs $(P,e)$ in which $P$ is an idempotent affine Rota-Baxter operator, $e\in\operatorname{Im}P$ is idempotent, and $P$ is split at $e$.
\end{enumerate}
Under this correspondence, $\mathfrak a_+=\operatorname{Im}P$ and $\mathfrak a_-=\operatorname{Ker}_eP$, while $P$ is the affine projection onto $\mathfrak a_+$ along $\mathfrak a_-$. Moreover, $Q_e(x)=\langle x,P(x),e\rangle$ is also an idempotent affine Rota-Baxter operator, split at $e$, and corresponds to the reversed factorization $\mathfrak a=\mathfrak a_-\oplus_e\mathfrak a_+$.
\end{theorem}

\begin{proof}
Start with a pair $(P,e)$ in (b). Its image is a Lie subaffgebra by Proposition~\ref{p8}, and its fibre over $e$ is a Lie subaffgebra because $P$ is split at $e$. If $a$ belongs to both factors, then $P(a)=a$ and $P(a)=e$, so $a=e$. Lemma~\ref{l3} gives the unique decomposition
$$
x=\langle P(x),e,Q_e(x)\rangle,
\quad P(x)\in\operatorname{Im}P,\quad Q_e(x)\in\operatorname{Ker}_eP.
$$
Hence these factors give an ordered affine Lie factorization.

Conversely, let $\mathfrak a=\mathfrak a_+\oplus_e\mathfrak a_-$ be an ordered affine Lie factorization. Since both factors contain $e$ and are closed under the bracket,
$$
\{e,e\}\in\mathfrak a_+\cap\mathfrak a_-=\{e\}.
$$
Choose $e$ as origin and write $\mathfrak a=\mathfrak a(\mathfrak g;\kappa,\lambda,0)$, where $\mathfrak g=T_e\mathfrak a$. The two factors become vector subspaces $\mathfrak g_+,\mathfrak g_-$ with $\mathfrak g=\mathfrak g_+\oplus\mathfrak g_-$. By Corollary~\ref{c1}, both are Lie subalgebras and are preserved by $\kappa$ and $\lambda$.

Let $R$ be the projection onto $\mathfrak g_+$ along $\mathfrak g_-$. For $x=x_++x_-$ and $y=y_++y_-$, with $x_\pm,y_\pm\in\mathfrak g_\pm$, we have
\begin{align*}
R([R(x),y]+[x,R(y)]-[x,y])
=R([x_+,y_+]-[x_-,y_-])
=[x_+,y_+]
=[R(x),R(y)].
\end{align*}
Here $R$ fixes $\mathfrak g_+$ and vanishes on $\mathfrak g_-$, and both summands are Lie subalgebras. Preservation of the summands also gives
$$
R\kappa(x_++x_-)=\kappa(x_+)=\kappa R(x_++x_-),
\quad R\lambda(x_++x_-)=\lambda(x_+)=\lambda R(x_++x_-).
$$
Thus $R^2=R$, $R\kappa R=\kappa R$ and $R\lambda R=\lambda R$.

Regard $R$ as the affine projection $P$ fixing $e$. Its affine Rota-Baxter identity can now be checked directly:
\begin{align*}
P(\{x,y\}_P)
&=R([x,y]_{R,-1})+R\kappa R(x)+R\lambda R(y-x)
\quad\text{(by Eq.~(\ref{e10}))}\\
&=[R(x),R(y)]+\kappa R(x)+\lambda(R(y)-R(x))\\
&=\{P(x),P(y)\}
\quad\text{(by Eq.~(\ref{e2}) with $s=p=0$)}.
\end{align*}
Moreover, $P^2=P$, $P(e)=e$, $\operatorname{Im}P=\mathfrak a_+$ and $\operatorname{Ker}_eP=\mathfrak a_-$. In particular, $P$ is split at $e$, so $(P,e)$ belongs to (b).

The projection associated with the factorization sends the unique expression $\langle a,e,b\rangle$ to $a$. Starting with $(P,e)$, the decomposition from Lemma~\ref{l3} therefore recovers $P$ itself. Starting with a factorization, the image and fibre of its projection recover its two factors. This proves that the constructions are inverse.

Finally, in the coordinates above,
$$
Q_e(x_++x_-)=x_++x_--R(x_++x_-)=x_-.
$$
It is the projection for the reversed factorization. Applying the construction just proved to that ordered factorization shows that $Q_e$ is an idempotent affine Rota-Baxter operator split at $e$.
\end{proof}

\begin{remark}
Theorem~\ref{t6} extends Proposition~\ref{p2} to affine spaces. The intersection point $e$ of the two factors determines the complementary projection $Q_e(x)=\langle x,P(x),e\rangle$. Choosing $e$ as origin gives the usual formula $\operatorname{id}-R$.
\end{remark}

\subsection{Affine doubles at nonzero weight}\label{u15}

Fix a nonzero weight. We first decompose the product affine space into a graph and the diagonal, then determine when the graph is closed under the componentwise bracket. Under these conditions, the two subaffgebras give a factorization whose tangent Lie algebra is the Rota-Baxter double of~\cite{LS2023}.

The product $\mathfrak a\times\mathfrak a$ carries the componentwise heap operation
\begin{align*}
\big\langle(a_1,a_2),(b_1,b_2),(c_1,c_2)\big\rangle
&=\big(\langle a_1,b_1,c_1\rangle,
\langle a_2,b_2,c_2\rangle\big).
\end{align*}
The affine action is also taken componentwise. All heap expressions in the product below refer to this operation.

Fix $\theta\in\mathbf{k}^{\times}$, let $\mathfrak a$ be an affine space, let $P:\mathfrak a\to\mathfrak a$ be affine, and fix $e\in\mathfrak a$. Define
$$
P^+_{\theta,e}(x)=\left\langle P(x),e,(\theta,e,x)\right\rangle,
\quad
H_e(x)=\big(P^+_{\theta,e}(x),P(x)\big).
$$
Write $\Gamma_e(P)=\operatorname{Im}H_e$ for the affine graph, $\Delta(\mathfrak a)=\{(x,x)\mid x\in\mathfrak a\}$ for the diagonal affine subspace, and $d_e=(P(e),P(e))$. Throughout this subsection, the fixed weight $\theta$ is suppressed in $H_e$ and $\Gamma_e(P)$.

\begin{prop}\label{p11}
For every affine map $P:\mathfrak a\to\mathfrak a$ and every $\theta\neq0$,
$
\mathfrak a\times\mathfrak a
=
\Gamma_e(P)\oplus_{d_e}\Delta(\mathfrak a)
$
as affine spaces. More precisely, for every $(u,v)\in\mathfrak a\times\mathfrak a$ there are unique $x,y\in\mathfrak a$ such that
\begin{equation}\label{e48}
(u,v)
=
\left\langle
H_e(x),d_e,(y,y)
\right\rangle.
\end{equation}
They are given by
\begin{equation}\label{e49}
x=(\theta^{-1},e,\langle u,v,e\rangle),
\quad
y=\langle v,P(x),P(e)\rangle.
\end{equation}
\end{prop}

\begin{proof}
Choose $e$ as origin. Write $P(x)=R(x)+p$, where $p=P(e)$ in these coordinates. Then
$$
P^+_{\theta,e}(x)=(R+\theta\operatorname{id})(x)+p,
\quad
H_e(x)=\big((R+\theta\operatorname{id})(x)+p,R(x)+p\big).
$$
If $H_e(x)$ belongs to the diagonal, then $\theta x=0$, so $x=0=e$ and $H_e(x)=d_e$. Thus $\Gamma_e(P)\cap\Delta(\mathfrak a)=\{d_e\}$.

For $(u,v)\in\mathfrak a\times\mathfrak a$, Eq.~(\ref{e48}) becomes
$$
u=(R+\theta\operatorname{id})(x)+y,
\quad
v=R(x)+y.
$$
Subtracting gives $u-v=\theta x$, hence $x=\theta^{-1}(u-v)$. The second equation then gives $y=v-R(x)$. These are exactly the affine formulas in Eq.~(\ref{e49}). The same calculation proves uniqueness.
\end{proof}

Now suppose that $\mathfrak a$ is a Lie affgebra. The product affine space carries the componentwise bracket
\begin{equation}\label{e50}
\{(a_1,a_2),(b_1,b_2)\}_{\times}
=
\big(\{a_1,b_1\},\{a_2,b_2\}\big).
\end{equation}
The following theorem characterizes closure of $\Gamma_e(P)$ under the componentwise bracket.

\begin{theorem}\label{t7}
Let $P:\mathfrak a\to\mathfrak a$ be a tangent Rota-Baxter operator of nonzero weight $\theta$, and fix $e\in\mathfrak a$. Choose $e$ as origin and write $\mathfrak a=\mathfrak a(\mathfrak g;\kappa_e,\lambda_e,s_e)$ and $P(x)=R(x)+p_e$, where $R$ is the linear part of $P$ and $p_e=P(e)$ in these coordinates. Then $\Gamma_e(P)$ is a Lie subaffgebra of $(\mathfrak a\times\mathfrak a,\{-,-\}_{\times})$ if and only if
\begin{equation}\label{e52}
R\kappa_e=\kappa_eR,\quad R(\lambda_e+\operatorname{ad}_{p_e})=(\lambda_e+\operatorname{ad}_{p_e})R, \quad s_e+\kappa_e(p_e)=p_e.
\end{equation}
The last equality is equivalent to the statement that $P(e)$ is an idempotent point of the original Lie affgebra.
\end{theorem}

\begin{proof}
With $e$ as origin, put $p_e=P(e)$ and let $R$ be the linear part of $P$, so that $P(x)=R(x)+p_e$. By assumption, $R$ satisfies Eq.~(\ref{e6}) at weight $\theta$. Since $P^+_{\theta,e}(x)=P(x)+\theta x$ in these coordinates, Eq.~(\ref{e2}) gives
\begin{align*}
&\{P^+_{\theta,e}(x),P^+_{\theta,e}(y)\}-\{P(x),P(y)\}\\
=&\ \theta[P(x),y]+\theta[x,P(y)]+\theta^2[x,y]
+\theta\kappa_e(x)+\theta\lambda_e(y-x)\\
=&\ \theta\bigl([R(x),y]+[x,R(y)]+\theta[x,y]
+[p_e,y]+[x,p_e]+\kappa_e(x)+\lambda_e(y-x)\bigr).
\end{align*}
Using Eq.~(\ref{e7}) and $[p_e,y]+[x,p_e]=\operatorname{ad}_{p_e}(y-x)$, we obtain
\begin{equation}\label{e53}
\{P^+_{\theta,e}(x),P^+_{\theta,e}(y)\}-\{P(x),P(y)\}
=\theta\bigl([x,y]_{R,\theta}+\kappa_e(x)
+(\lambda_e+\operatorname{ad}_{p_e})(y-x)\bigr).
\end{equation}
For every $z$, the difference between the two components of $H_e(z)$ is $\theta z$. Thus $H_e$ is injective, and Eq.~(\ref{e53}) forces the preimage of the product bracket, if it exists, to be
\begin{equation}\label{e54}
\{x,y\}^{\diamond}_{P,e}
=[x,y]_{R,\theta}+\kappa_e(x)+(\lambda_e+\operatorname{ad}_{p_e})(y-x).
\end{equation}
Because Eq.~(\ref{e53}) already gives the required difference of components, both components agree with those of $H_e(\{x,y\}^{\diamond}_{P,e})$ precisely when the second components agree. Hence graph closure is equivalent to
\begin{equation}\label{e55}
P\big(\{x,y\}^{\diamond}_{P,e}\big)=\{P(x),P(y)\}.
\end{equation}

We now compute the difference in Eq.~(\ref{e55}):
\begin{align*}
P\big(\{x,y\}^{\diamond}_{P,e}\big)-\{P(x),P(y)\}
&=R([x,y]_{R,\theta})-[R(x),R(y)]
+(R\kappa_e-\kappa_eR)(x)\\
&\quad+\bigl(R(\lambda_e+\operatorname{ad}_{p_e})
-(\lambda_e+\operatorname{ad}_{p_e})R\bigr)(y-x)
+p_e-s_e-\kappa_e(p_e)\\
&=(R\kappa_e-\kappa_eR)(x)
+\bigl(R(\lambda_e+\operatorname{ad}_{p_e})
-(\lambda_e+\operatorname{ad}_{p_e})R\bigr)(y-x)\\
&\quad+p_e-s_e-\kappa_e(p_e)
\quad\text{(by Eq.~(\ref{e6}))}.
\end{align*}
If this is zero for all $x,y$, setting $x=y=0$ gives $s_e+\kappa_e(p_e)=p_e$. After removing the constant term, setting $y=x$ gives $R\kappa_e=\kappa_eR$, and setting $x=0$ gives the remaining commutation relation. These are exactly Eq.~(\ref{e52}). Conversely, those three identities make the displayed difference zero, proving graph closure. Finally,
$$
\{P(e),P(e)\}=\{p_e,p_e\}=\kappa_e(p_e)+s_e,
$$
so the last condition is equivalent to idempotency of $P(e)$.
\end{proof}

\begin{coro}\label{c8}
Under the equivalent conditions of Theorem~\ref{t7}, the operation in Eq.~(\ref{e54}) is a Lie affgebra bracket on $\mathfrak a$, and $H_e:(\mathfrak a,\{-,-\}^{\diamond}_{P,e})\to\Gamma_e(P)$ is a Lie affgebra isomorphism. With $e$ as origin, this Lie affgebra has the standard form
\begin{equation}\label{e56}
\mathfrak a
\big(
(\mathfrak g,[-,-]_{R,\theta});
\kappa_e,
\lambda_e+\operatorname{ad}_{p_e},
0
\big).
\end{equation}
\end{coro}

\begin{proof}
The map $H_e$ is an affine bijection from $\mathfrak a$ onto $\Gamma_e(P)$. By Theorem~\ref{t7}, the latter is a Lie subaffgebra of the product Lie affgebra. Transporting the bracket along $H_e$ gives a Lie affgebra bracket on $\mathfrak a$, and the computation in Eq.~(\ref{e53}) shows that it is given by Eq.~(\ref{e54}). Eq.~(\ref{e56}) follows from the same computation.
\end{proof}

The bracket $\{-,-\}^{\diamond}_{P,e}$ induced by $H_e$ need not equal the weighted descendant $\{-,-\}_{P,\theta}$ from Definition~\ref{d2}. For example, on the canonical Lie affgebra with bracket $\{x,y\}=[x,y]+y$, take $P(x)=0$, $e=0$ and $\theta=-1$. The conditions of Theorem~\ref{t7} hold, but
$$
\{x,y\}^{\diamond}_{P,0}=-[x,y]+y,
\quad
\{x,y\}_{P,-1}=-[x,y].
$$
The first factor in the double constructed below carries the bracket $\{-,-\}^{\diamond}_{P,e}$.

\begin{theorem}\label{t8}
Assume the hypotheses and equivalent conditions of Theorem~\ref{t7}. Define
$$
\Phi_e:\mathfrak a\times\mathfrak a\longrightarrow\mathfrak a\times\mathfrak a,
\quad
\Phi_e(x,y)
=
\left\langle
H_e(x),d_e,(y,y)
\right\rangle.
$$
Then $\Phi_e$ is an affine isomorphism. Pulling back the componentwise bracket in Eq.~(\ref{e50}) along $\Phi_e$ gives a Lie affgebra bracket on $\mathfrak a\times\mathfrak a$ for which $\mathfrak a\times\{P(e)\}$ and $\{e\}\times\mathfrak a$ are Lie subaffgebras. The first is isomorphic to $(\mathfrak a,\{-,-\}^{\diamond}_{P,e})$ and the second to the original Lie affgebra $\mathfrak a$. We call the resulting Lie affgebra on $\mathfrak a\times\mathfrak a$ the \textup{\textbf{affine double}} associated with $P$, $e$ and $\theta$. As a Lie affgebra, it is isomorphic to the original product by $\Phi_e$. The specified graph and diagonal factors are part of this construction.

Identifying the tangent spaces at $(e,P(e))$ and $d_e$ with $\mathfrak g\oplus\mathfrak g$ by translation, the linear part $\vec\Phi_e$ is
\begin{equation}\label{e58}
\vec\Phi_e(\xi,x)
=
\big((R+\theta\operatorname{id})(\xi)+x,\,R(\xi)+x\big).
\end{equation}
Here $\xi,x\in\mathfrak g$, and $\vec\Phi_e$ denotes the linear part in the sense of Section~\ref{u5}.
\end{theorem}

\begin{proof}
Choose $e$ as origin, and write $P(x)=R(x)+p_e$, where $p_e=P(e)$. Then
\begin{align*}
\Phi_e(x,y)
&=\big((R+\theta\operatorname{id})(x)+p_e,\ R(x)+p_e\big)
-(p_e,p_e)+(y,y)\\
&=\big((R+\theta\operatorname{id})(x)+y,\ R(x)+y\big).
\end{align*}
Since $\theta\neq0$, its inverse is
$$
\Phi_e^{-1}(u,v)
=\bigl(\theta^{-1}(u-v),\ v-R(\theta^{-1}(u-v))\bigr).
$$
Indeed, the difference of the two coordinates determines $x=\theta^{-1}(u-v)$, and then the second determines $y=v-R(x)$. Thus $\Phi_e$ and its inverse are affine.

For $x,y,z,w\in\mathfrak a$, define the transported bracket by
$$
\{(x,y),(z,w)\}_{\Phi_e}
=\Phi_e^{-1}\bigl(\{\Phi_e(x,y),\Phi_e(z,w)\}_{\times}\bigr).
$$
This is a Lie affgebra bracket: $\Phi_e$ preserves the affine operations, and applying $\Phi_e$ to either defining identity reduces it to that identity for the componentwise bracket.

The two restrictions of $\Phi_e$ are
$$
\Phi_e(x,p_e)=H_e(x),\quad \Phi_e(0,y)=(y,y).
$$
By Theorem~\ref{t7} and Corollary~\ref{c8}, the graph bracket is transported by $H_e$. Therefore
\begin{align*}
\{(x,p_e),(y,p_e)\}_{\Phi_e}
&=\Phi_e^{-1}\bigl(\{H_e(x),H_e(y)\}_{\times}\bigr)\\
&=\Phi_e^{-1}\bigl(H_e(\{x,y\}^{\diamond}_{P,e})\bigr)
\quad\text{(by Corollary~\ref{c8})}\\
&=(\{x,y\}^{\diamond}_{P,e},p_e).
\end{align*}
Likewise,
\begin{align*}
\{(0,x),(0,y)\}_{\Phi_e}
&=\Phi_e^{-1}\bigl(\{(x,x),(y,y)\}_{\times}\bigr)\\
&=\Phi_e^{-1}\bigl(\{x,y\},\{x,y\}\bigr)
\quad\text{(by Eq.~(\ref{e50}))}\\
&=(0,\{x,y\}).
\end{align*}
These formulas prove closure of both factors and identify their brackets.

Finally, use $(0,p_e)$ as the origin in the domain and $(p_e,p_e)=d_e$ as the origin in the codomain. For $\xi,x\in\mathfrak g$, we have
\begin{align*}
\vec\Phi_e(\xi,x)
&=\Phi_e(\xi,p_e+x)-(p_e,p_e)
=\big((R+\theta\operatorname{id})(\xi)+x,\ R(\xi)+x\big),
\end{align*}
which is Eq.~(\ref{e58}).
\end{proof}

By Eq.~(\ref{e58}), the tangent Lie algebra of the affine double is the Rota-Baxter double of nonzero weight associated with $(\mathfrak g,R)$ in~\cite[Proposition~3.2]{LS2023}. In particular, its two tangent factors form the matched pair of~\cite[Theorem~3.3]{LS2023}, with mutual actions
$$
\rho:(\mathfrak g,[-,-]_{R,\theta})\longrightarrow\operatorname{End}_{\mathbf{k}}(\mathfrak g),
\quad
\mu:(\mathfrak g,[-,-])\longrightarrow\operatorname{End}_{\mathbf{k}}(\mathfrak g)
$$
given by
\begin{equation}\label{e59}
\rho(\xi)(x)=R[x,\xi]-[x,R\xi],
\quad
\mu(x)(\xi)=[x,\xi].
\end{equation}
Here $\xi,x\in\mathfrak g$. To compute the mixed bracket, apply Eq.~(\ref{e58}) to the two factors and then use its inverse:
\begin{align*}
[(\xi,0),(0,x)]
&=\vec\Phi_e^{-1}\bigl([(R+\theta\operatorname{id})(\xi),x],[R(\xi),x]\bigr)\\
&=\bigl([\xi,x],\ [R(\xi),x]-R([\xi,x])\bigr)\\
&=\bigl(-\mu(x)(\xi),\rho(\xi)(x)\bigr)
\quad\text{(by Eq.~(\ref{e59}))}.
\end{align*}

The identity in Eq.~(\ref{e25}) implies the ordinary Rota-Baxter identity for the linear part by Eq.~(\ref{e27}). For an idempotent affine Rota-Baxter operator, graph closure is equivalent to splitting.

\begin{prop}\label{p12}
Let $P$ be an idempotent affine Rota-Baxter operator, and let $e\in\operatorname{Im}P$ be an idempotent point. Take $\theta=-1$ in the graph construction. Then $P(e)=e$, and the conditions in Eq.~(\ref{e52}) are equivalent to the splitting conditions of Theorem~\ref{t5}. When they hold,
\begin{equation}\label{e60}
\Gamma_e(P)
=
\operatorname{Ker}_eP\times\operatorname{Im}P.
\end{equation}
Thus splitting at weight $-1$ is the projection case of the affine double construction.
\end{prop}

\begin{proof}
Since $e\in\operatorname{Im}P$ and $P^2=P$, we have $P(e)=e$. Choose $e$ as origin. Then $p_e=s_e=0$, and $P$ becomes a linear projection $R$. Hence Eq.~(\ref{e52}) reduces to $R\kappa_e=\kappa_eR$ and $R\lambda_e=\lambda_eR$, which is condition (iii) of Theorem~\ref{t5}.

At weight $-1$, $P^+_{-1,e}(x)=(R-\operatorname{id})(x)$. Writing $x=x_++x_-$ with $x_+\in\operatorname{Im}R$ and $x_-\in\ker R$ gives $H_e(x)=(-x_-,x_+)$.
Since $x_+$ and $x_-$ vary independently, the image of $H_e$ is $\ker R\times\operatorname{Im}R$, which gives Eq.~(\ref{e60}) in affine notation.
\end{proof}

\section{Explicit operators and classification}\label{s4}

We first consider Lie affgebras whose two structural endomorphisms are equal scalar multiples of the identity. The coordinate equations then give conditions for affine Rota-Baxter operators and their splitting. For canonical Lie affgebras, all operators are described by Lie subalgebra decompositions and translation vectors. This description gives the classification in the abelian case of arbitrary finite dimension and in the nonabelian case of dimension 2.

\subsection{Scalar structural endomorphisms}\label{u16}

Fix a Lie algebra $(\mathfrak g,[-,-])$. For each pair $(\gamma,s)\in\mathbf{k}\times\mathfrak g$, set
$$
\mathfrak a_{\gamma,s}(\mathfrak g)
=\mathfrak a(\mathfrak g;\gamma\operatorname{id},\gamma\operatorname{id},s).
$$
Thus the affine space underlying $\mathfrak g$ has bracket
\begin{equation}\label{esc}
\{x,y\}=[x,y]+\gamma y+s,\quad x,y\in\mathfrak g.
\end{equation}
The Lie algebra $\mathfrak g$ and its bracket remain fixed, while the scalar $\gamma$ and the constant vector $s$ are the parameters. Each choice satisfies Eq.~(\ref{e1}), so Proposition~\ref{tstr} gives a Lie affgebra. Its structural endomorphisms are $\kappa=\lambda=\gamma\operatorname{id}$. The case $\gamma=1$ and $s=0$ is the canonical Lie affgebra considered in Section~\ref{u17}. Such brackets also occur in the classification of Lie affgebras with simple tangent Lie algebras~\cite{ABR2025}.

For an affine map $P:\mathfrak a_{\gamma,s}(\mathfrak g)\to\mathfrak a_{\gamma,s}(\mathfrak g)$, write $P(x)=R(x)+p$, where $R$ is its linear part and $p=P(0)$.

\begin{lemma}\label{l4}
The affine map $P=R+p$ satisfies $P^2=P$ if and only if $R^2=R$ and $R(p)=0$.
\end{lemma}

\begin{proof}
We have $P^2(x)=R^2(x)+R(p)+p$.
Comparing the linear and constant parts with $P(x)=R(x)+p$ gives the assertion.
\end{proof}

\begin{prop}\label{p13}
Assume $\gamma\neq0$. An affine map $P(x)=R(x)+p$ is an affine Rota-Baxter operator on $\mathfrak a_{\gamma,s}(\mathfrak g)$ if and only if
\begin{enumerate}[label={\rm(\roman*)}]
\item $R$ is a Rota-Baxter operator of weight $-1$.
\item $R^2=R$.
\item $[R,\operatorname{ad}_p]=0$.
\item
\begin{equation}\label{e62}
R(s+\gamma p)=s+(\gamma-1)p.
\end{equation}
\end{enumerate}
Every operator satisfying these conditions is idempotent.
\end{prop}

\begin{proof}
Apply Corollary~\ref{c2} with $\kappa=\lambda=\gamma\operatorname{id}$. Its bilinear condition is (i), and $R\kappa R=\kappa R$ becomes $R^2=R$ because $\gamma\neq0$. With this equality, the first identity in Eq.~(\ref{e17}) reduces to $R\operatorname{ad}_p=\operatorname{ad}_pR$, and the second becomes Eq.~(\ref{e62}). This proves the equivalence.

Applying $R$ to Eq.~(\ref{e62}) and using $R^2=R$ gives
$
R(s)+\gamma R(p)=R(s)+(\gamma-1)R(p),
$
so $R(p)=0$. Lemma~\ref{l4} now shows that $P^2=P$.
\end{proof}

For $\gamma=0$, Corollary~\ref{c2} states that $P$ is an affine Rota-Baxter operator if and only if $R$ is Rota-Baxter of weight $-1$, $R\operatorname{ad}_p=\operatorname{ad}_pR$, and $R(s)=s-p$. There is no projection condition in this case. For instance, if $\mathfrak g$ is abelian and $s=p=0$, every linear endomorphism $R$ satisfies the affine Rota-Baxter identity. Thus the hypothesis $\gamma\neq0$ in Proposition~\ref{p13} is necessary for its idempotency conclusion.

For nonzero $\gamma$, the condition for splitting becomes the following.

\begin{coro}\label{c9}
Let $P(x)=R(x)+p$ be an affine Rota-Baxter operator on $\mathfrak a_{\gamma,s}(\mathfrak g)$ with $\gamma\neq0$. Assume that $p=P(0)$ is an idempotent point, equivalently $s+(\gamma-1)p=0$.
Then
$$
\mathfrak a_{\gamma,s}(\mathfrak g)
=
(p+\operatorname{Im}R)
\oplus_p
(p+\ker R)
$$
is an affine Lie factorization. In particular, $P$ is split at $p$.
\end{coro}

\begin{proof}
By Proposition~\ref{p13}, $P$ is idempotent, so Lemma~\ref{l4} gives $R^2=R$ and $R(p)=0$.
Hence $p\in\ker R$ and $p+\ker R=\ker R$ in the chosen coordinates. Proposition~\ref{p13} gives $[R,\operatorname{ad}_p]=0$.
Thus $\operatorname{ad}_p$ preserves both $\operatorname{Im}R$ and $\ker R$.

Apply Proposition~\ref{p1} with $a=p$. The condition $\{p,p\}=p$ gives $\gamma p-p+s=0$.
The map $\kappa=\gamma\operatorname{id}$ preserves every vector subspace, and $\lambda+\operatorname{ad}_p=\gamma\operatorname{id}+\operatorname{ad}_p$ preserves both $\operatorname{Im}R$ and $\ker R$. Since the two spaces are Lie subalgebras by Proposition~\ref{p2}, both $p+\operatorname{Im}R$ and $p+\ker R$ are Lie subaffgebras. Lemma~\ref{l3} gives the affine direct sum decomposition.
\end{proof}

\subsection{Canonical Lie affgebras}\label{u17}

In the notation of Section~\ref{u16}, the canonical Lie affgebra associated with $\mathfrak g$ is $\mathfrak a_{\mathrm{can}}(\mathfrak g)=\mathfrak a_{1,0}(\mathfrak g)$. Its structure parameters are $\kappa=\lambda=\operatorname{id}$ and $s=0$, and its bracket is $\{x,y\}=[x,y]+y$. Every point is idempotent, since $\{p,p\}=p$.

\begin{theorem}\label{t9}
Let $\mathfrak g=\mathfrak g_+\oplus\mathfrak g_-$ be a vector space direct sum in which $\mathfrak g_+$ and $\mathfrak g_-$ are Lie subalgebras. Let $p\in\mathfrak g_-$ satisfy
\begin{equation}\label{e65}
[p,\mathfrak g_+]
\subseteq
\mathfrak g_+.
\end{equation}
Define
\begin{equation}\label{e66}
P(x_++x_-)=x_++p,
\quad
x_+\in\mathfrak g_+,
\quad
x_-\in\mathfrak g_-.
\end{equation}
Then $P$ is an idempotent affine Rota-Baxter operator on $\mathfrak a_{\mathrm{can}}(\mathfrak g)$, split at $p$, and
\begin{equation}\label{e67}
\mathfrak a_{\mathrm{can}}(\mathfrak g)
=
(p+\mathfrak g_+)
\oplus_p
\mathfrak g_-.
\end{equation}
Conversely, every affine Rota-Baxter operator on $\mathfrak a_{\mathrm{can}}(\mathfrak g)$ is of the form in Eq.~(\ref{e66}) for a decomposition into Lie subalgebras and an element $p\in\mathfrak g_-$ satisfying Eq.~(\ref{e65}). In particular, every affine Rota-Baxter operator on $\mathfrak a_{\mathrm{can}}(\mathfrak g)$ is idempotent and split at $p$.
\end{theorem}

\begin{proof}
For the canonical Lie affgebra, $\kappa=\lambda=\operatorname{id}$ and $s=0$. Substituting these values into Corollary~\ref{c2} shows that $P=R+p$ is an affine Rota-Baxter operator exactly when
\begin{align*}
&R([x,y]_{R,-1})=[R(x),R(y)],\quad R^2=R,\\
&R(R+\operatorname{ad}_p)=(\operatorname{id}+\operatorname{ad}_p)R,
\quad R(p)=0.
\end{align*}
Under $R^2=R$, the third equality is $R\operatorname{ad}_p=\operatorname{ad}_pR$.

Now start with the decomposition and vector in the statement, and let $R$ be the projection onto $\mathfrak g_+$ along $\mathfrak g_-$. Proposition~\ref{p2} gives the first two identities above. Since $p\in\mathfrak g_-$, we have $R(p)=0$. The hypothesis $[p,\mathfrak g_+]\subseteq\mathfrak g_+$ and the subalgebra condition on $\mathfrak g_-$ give, for $x_\pm\in\mathfrak g_\pm$,
\begin{align*}
R\operatorname{ad}_p(x_++x_-)
=R([p,x_+]+[p,x_-])
=[p,x_+]
=\operatorname{ad}_pR(x_++x_-).
\end{align*}
Thus all four coordinate identities hold, and $P$ is an affine Rota-Baxter operator. Moreover,
$$
P^2(x)=R^2(x)+R(p)+p=R(x)+p=P(x),\quad P(p)=p.
$$

Its image is $p+\mathfrak g_+$, and its fibre over $p$ is $\mathfrak g_-$. For $x_\pm,y_\pm\in\mathfrak g_\pm$, the canonical bracket gives
\begin{align*}
\{p+x_+,p+y_+\}-p
&=[x_+,y_+]+[p,y_+-x_+]+y_+\in\mathfrak g_+,\\
\{x_-,y_-\}
&=[x_-,y_-]+y_-\in\mathfrak g_-.
\end{align*}
Both factors are therefore Lie subaffgebras. Since $p\in\mathfrak g_-$, their intersection is $\{p\}$, and every $x=x_++x_-$ has the unique expression
$
x=\langle p+x_+,p,x_-\rangle.
$
This proves Eq.~(\ref{e67}) and splitting at $p$.

Conversely, suppose $P=R+p$ is an affine Rota-Baxter operator. The coordinate identities above make $R$ an idempotent ordinary Rota-Baxter operator of weight $-1$. Set $\mathfrak g_+=\operatorname{Im}R$ and $\mathfrak g_-=\ker R$. Proposition~\ref{p2} gives a direct sum of Lie subalgebras, and $R(p)=0$ gives $p\in\mathfrak g_-$. If $x_+\in\mathfrak g_+$, then
$
R([p,x_+])=[p,R(x_+)]=[p,x_+],
$
so $[p,x_+]\in\mathfrak g_+$. Thus Eq.~(\ref{e65}) holds, and $P(x_++x_-)=x_++p$. The preceding construction proves idempotency and splitting for this operator.
\end{proof}

The complementary affine Rota-Baxter projection is $Q_p(x_++x_-)=x_-$. The common point of the two affine factors is $p$.

\subsection{Classification in low dimensions}\label{u18}

We classify affine Rota-Baxter operators on canonical Lie affgebras of dimension at most 2. The description of automorphisms below holds in arbitrary dimension, and the abelian classification holds in every finite dimension. Two affine Rota-Baxter operators $P$ and $P'$ on $\mathfrak a_{\mathrm{can}}(\mathfrak g)$ are called \textbf{equivalent}, written $P\sim P'$, if $\varphi P=P'\varphi$ for some automorphism $\varphi$ of this Lie affgebra. This equivalence acts on both the linear part and the translation vector of the operator.

\begin{prop}\label{p14}
Every Lie affgebra automorphism of $\mathfrak a_{\mathrm{can}}(\mathfrak g)$ has the unique form
$$
\varphi(x)=\psi(x)+q,
\quad
\psi\in\operatorname{Aut}_{\rm Lie}(\mathfrak g),
\quad
q\in Z(\mathfrak g).
$$
Here $\operatorname{Aut}_{\rm Lie}(\mathfrak g)$ denotes the group of Lie algebra automorphisms of $\mathfrak g$, and $Z(\mathfrak g)$ is its center.
Consequently, if $P=R+p$ and $P'=R'+p'$, then $P$ and $P'$ are equivalent if and only if there are $\psi\in\operatorname{Aut}_{\rm Lie}(\mathfrak g)$ and $q\in Z(\mathfrak g)$ such that
\begin{equation}\label{e68}
R'=\psi R\psi^{-1},
\quad
p'=\psi(p)+(\operatorname{id}-R')q.
\end{equation}
\end{prop}

\begin{proof}
Apply the homomorphism conditions in Eq.~(\ref{e3}) to $\kappa=\lambda=\operatorname{id}$ and $s=0$. If $\varphi(x)=\psi(x)+q$ is an automorphism, then $\psi$ is a Lie algebra automorphism, and the second equation in Eq.~(\ref{e3}) becomes $\psi=(\operatorname{ad}_q+\operatorname{id})\psi$. Since $\psi$ is surjective, this is equivalent to $\operatorname{ad}_q=0$, or $q\in Z(\mathfrak g)$. The first and third equations in Eq.~(\ref{e3}) hold automatically. Conversely, every Lie algebra automorphism $\psi$ and central vector $q$ satisfy the homomorphism conditions and define a bijective affine map. Comparing the linear and constant parts of $\varphi P=P'\varphi$ gives $\psi R=R'\psi$ and $\psi(p)+q=R'(q)+p'$, which is equivalent to Eq.~(\ref{e68}).
\end{proof}

\begin{theorem}\label{t10}
Let $\mathfrak g$ be abelian of finite dimension $n\geq0$, with basis $e_1,\ldots,e_n$. The equivalence classes of affine Rota-Baxter operators on $\mathfrak a_{\mathrm{can}}(\mathfrak g)$ are indexed by $r=0,\ldots,n$. A representative of the class indexed by $r$ is
\begin{equation}\label{e69}
P_r\left(\sum_{i=1}^n x_i e_i\right)
=
\sum_{i=1}^r x_i e_i.
\end{equation}
Here $x_1,\ldots,x_n\in\mathbf{k}$.
For $n=0$, the empty sum denotes the unique point and there is one class. There are two classes in dimension 1 and three classes in the abelian case of dimension 2.
\end{theorem}

\begin{proof}
Theorem~\ref{t9} gives $P=R+p$, where $R^2=R$ and $p\in\ker R$. Conversely, these conditions suffice because every subspace of an abelian Lie algebra is a Lie subalgebra and Eq.~(\ref{e65}) holds. In Proposition~\ref{p14}, take $\psi=\operatorname{id}$ and $q=-p$. Since $Z(\mathfrak g)=\mathfrak g$, this is an allowed automorphism. By Eq.~(\ref{e68}), the conjugate operator has the same linear part $R$ and constant term
$
p+(\operatorname{id}-R)(-p)=R(p)=0.
$

Choose a basis of $\operatorname{Im}R$ followed by a basis of $\ker R$. The map $R$ fixes the basis vectors in its image and sends those in its kernel to zero. Thus, with $r=\operatorname{rank}R$, it is conjugate to the projection in Eq.~(\ref{e69}). Every displayed representative is a projection and therefore an affine Rota-Baxter operator. Finally, Eq.~(\ref{e68}) shows that any equivalence conjugates the linear parts, so it preserves their ranks. Thus every operator is equivalent to exactly one of the displayed representatives.
\end{proof}

Let $\mathfrak r_2$ denote the Lie algebra with basis $h,e$ and bracket $[h,e]=e$. Every nonabelian Lie algebra of dimension 2 over $\mathbf{k}$ is isomorphic to $\mathfrak r_2$. Indeed, if $x,y$ is a basis and $e=[x,y]=ax+by\neq0$, choose $c,d\in\mathbf{k}$ with $cb-da=1$ and put $h=cx+dy$. Then $h,e$ is a basis and $[h,e]=e$. Thus no assumption that $\mathbf{k}$ is algebraically closed is needed below.

\begin{theorem}\label{t11}
Let $\mathfrak r_2$ have basis $h,e$ and bracket $[h,e]=e$. Every affine Rota-Baxter operator on $\mathfrak a_{\mathrm{can}}(\mathfrak r_2)$ is equivalent to exactly one of the following operators:
$$
\begin{array}{c|c|c}
\text{operator} & P(uh+ve) & \text{parameter} \\ \hline
C_0 & 0 & {} \\
C_e & e & {} \\
C_\alpha & \alpha h & \alpha\in\mathbf{k}^\times \\
I & uh+ve & {} \\
A_t & ve+th & t\in\mathbf{k} \\
B & uh & {} \\
D & (u-v)h & {}
\end{array}
$$
Here $u,v\in\mathbf{k}$, and the first column names the maps in the second column. The operators $C_\alpha$ are constant maps, and $A_t$ is projection onto $\mathbf{k}e$ followed by translation by $th$. Different rows are inequivalent. For $\alpha,\beta\in\mathbf{k}^\times$, we have $C_\alpha\sim C_\beta$ if and only if $\alpha=\beta$. For $t,t'\in\mathbf{k}$, we have $A_t\sim A_{t'}$ if and only if $t=t'$.
\end{theorem}

\begin{proof}
The derived ideal is $[\mathfrak r_2,\mathfrak r_2]=\mathbf{k}e$. An automorphism must therefore send $e$ to $de$ with $d\neq0$. Writing its value at $h$ as $ah+ce$ and preserving $[h,e]=e$ gives $ad=d$, hence $a=1$. Conversely, all these maps preserve the bracket. Thus
\begin{equation}\label{e70}
\operatorname{Aut}_{\rm Lie}(\mathfrak r_2)
=\{\psi_{c,d}:\psi_{c,d}(h)=h+ce,\
\psi_{c,d}(e)=de\mid c\in\mathbf{k},\ d\in\mathbf{k}^{\times}\}.
\end{equation}
If $z=ah+be$ is central, then $[z,e]=ae=0$ and $[z,h]=-be=0$, so $z=0$. By Proposition~\ref{p14}, equivalence is therefore given by
$
R'=\psi_{c,d}R\psi_{c,d}^{-1}, p'=\psi_{c,d}(p).
$

Write $P=R+p$ as in Theorem~\ref{t9}. Since $R$ is a projection, its rank is $0$, $1$ or $2$. At rank zero, $P$ is the constant map with value $p=\alpha h+\beta e$. Using Eq.~(\ref{e70}), we have
$$
\psi_{c,d}(\alpha h+\beta e)
=\alpha(h+ce)+\beta de
=\alpha h+(\alpha c+\beta d)e.
$$ If $p=0$, the result is $C_0$. If $\alpha=0$ and $\beta\neq0$, the choice $c=0$, $d=\beta^{-1}$ gives $C_e$. If $\alpha\neq0$, the choice $d=1$, $c=-\beta/\alpha$ gives $C_\alpha$. The coefficient $\alpha$ is invariant under every automorphism. These classes are therefore distinct, including distinct nonzero values of $\alpha$. At rank two, $R=\operatorname{id}$ and $p\in\ker R=0$, giving $I$.

At rank one, $(\mathfrak g_+,\mathfrak g_-)=(\operatorname{Im}R,\ker R)$ is an ordered pair of distinct lines. Every line is a Lie subalgebra. Besides $\mathbf{k}e$, every line is $\mathbf{k}(h+ae)$ for some $a\in\mathbf{k}$, and $\psi_{c,d}\bigl(\mathbf{k}(h+ae)\bigr)=\mathbf{k}(h+(c+da)e)$. If one line is $\mathbf{k}e$, the other is sent to $\mathbf{k}h$ by choosing $d=1$ and $c=-a$. If both lines are $\mathbf{k}(h+ae),\mathbf{k}(h+be)$ with $a\neq b$, take $d=(b-a)^{-1}$ and $c=-a(b-a)^{-1}$. The ordered pair is then sent to $(\mathbf{k}h,\mathbf{k}(h+e))$. Consequently there are precisely three orbits of ordered pairs:
$
(\mathbf{k}e,\mathbf{k}h),(\mathbf{k}h,\mathbf{k}e),(\mathbf{k}h,\mathbf{k}(h+e)).
$
They are distinct because every automorphism preserves $\mathbf{k}e$ and the order of the image and kernel.

For $(\mathbf{k}e,\mathbf{k}h)$, the allowable translation vectors are $p=th$ with arbitrary $t\in\mathbf{k}$, since $[th,e]=te$. This gives $A_t$. An automorphism preserving this ordered pair has $c=0$ and arbitrary $d\neq0$, so it fixes $h$ and cannot change $t$.

For $(\mathbf{k}h,\mathbf{k}e)$, write $p=te$. Then $[p,h]=-te$ belongs to $\mathbf{k}h$ only if $t=0$. The projection is $R(uh+ve)=uh$, giving $B$. For $(\mathbf{k}h,\mathbf{k}(h+e))$, write $p=t(h+e)$. Again $[p,h]=-te\in \mathbf{k}h$ forces $t=0$. The decomposition
$
uh+ve=(u-v)h+v(h+e)
$
gives the projection $R(uh+ve)=(u-v)h$, hence $D$.

Conversely, the constant maps in the table are affine Rota-Baxter operators by Proposition~\ref{p7}, since every point of the canonical Lie affgebra is idempotent. The identity map also satisfies Eq.~(\ref{e9}). For the three rows of rank one, the displayed complementary lines are Lie subalgebras, $p$ lies in the kernel, and the calculations above verify Eq.~(\ref{e65}). Theorem~\ref{t9} therefore proves existence for every listed parameter. Rank separates the cases of ranks zero, one and two. Within rank one, the three orbits of ordered pairs are distinct and the stabilizer calculation separates the parameters of $A_t$. Thus the table is complete and its representatives are pairwise inequivalent.
\end{proof}

Over $\mathbb R$, the ordinary Rota-Baxter operators of weight $1$ on the nonabelian Lie algebra of dimension 2 are described in~\cite[Example~4.3]{JSZ2025}, using a rescaled basis. Negating these operators gives the ordinary operators of weight $-1$. On the canonical Lie affgebra, the affine Rota-Baxter identity additionally requires $R^2=R$, $R(p)=0$ and $R\operatorname{ad}_p=\operatorname{ad}_pR$, where $P(x)=R(x)+p$. Theorem~\ref{t11} classifies the resulting affine operators, including their translation terms, up to Lie affgebra automorphism.

\medskip
\noindent\textbf{Acknowledgments.} This work is supported by the National Natural Science Foundation of China (12671029 and 12101316). A. Khudoyberdiyev thanks the Sino-Russian Mathematics Center at Peking University, Beijing, China, for its hospitality during his stay, when part of this work was carried out. Y. Zhang and Y. Xie acknowledge support from the visiting scholar program of the Chern Institute of Mathematics and thank Professor Chengming Bai for his hospitality. The authors thank Professor Li Guo for helpful discussions.

\begingroup\emergencystretch=1em
\noindent\textbf{Declaration of generative AI use.} During the preparation of this work, the authors used \mbox{ChatGPT (OpenAI)} to assist in checking the completeness of the classification presented in Section~\ref{u18}. The results obtained with its assistance agreed with the authors' calculations. The authors reviewed and edited the output as needed and take full responsibility for the content of the published article.\par\endgroup

\noindent\textbf{Declaration of interests.} The authors have no conflicts of interest to disclose.

\noindent\textbf{Data availability.} Data sharing is not applicable as no data were created or analyzed.

\end{document}